\documentclass[10pt,a4paper]{article}

\usepackage[latin1]{inputenc}
\usepackage[english]{babel}
\usepackage{mathrsfs}
\usepackage{amsmath}
\usepackage{amsfonts}
\usepackage{amssymb}
\usepackage{amsthm}
\usepackage{graphicx, color}

\newtheorem{te}{Theorem}

\newtheorem{os}{Remark}

\newtheorem{lem}{Lemma}
\newtheorem{coro}{Corollary}

\numberwithin{equation}{section}
\allowdisplaybreaks

\begin{document}

\title{Fractional Poisson process with random drift}
\author{Luisa Beghin\thanks{
Dipartimento di Scienze Statistiche, Sapienza University of Rome - P.le A.
Moro 5 - 00185, Rome, Italy. Email: luisa.beghin@uniroma1.it} \and \& \and %
Mirko D'Ovidio\thanks{%
Dipartimento di Scienze di Base e Applicate per l'Ingegneria, Sapienza
University of Rome - A. Scarpa 10 - 00161, Rome, Italy. Email:
mirko.dovidio@uniroma1.it}}
\maketitle

\begin{abstract}
We study the connection between PDEs and L\'{e}vy processes running with clocks given by time-changed Poisson processes with stochastic drifts. The random times we deal with are therefore given by time-changed Poissonian jumps related to some Frobenious-Perron operators $K$ associated to random translations. Moreover, we also consider their hitting times as a random clock. Thus, we study processes driven by equations involving time-fractional operators (modelling memory) and fractional powers of the difference operator $I-K$ (modelling jumps). For this large class of processes we also provide, in some cases, the explicit representation of the transition probability laws. To this aim, we show that a special role is played by the translation operator associated to the representation of the Poisson semigroup.
\end{abstract}


\textbf{Subjclass :} 60G35, 60G50

\section{Introduction and preliminary results}

The aim of this paper is to study a real-valued version of the
Poisson process, defined as follows%
\begin{equation}
N(t)+a\mathfrak{F}_{t}^{\alpha ,\beta },\quad t>0,\;a\geq 0,\;\alpha ,\beta
\in (0,1],  \label{dr}
\end{equation}%
where $N$ denotes the standard Poisson process. The drift is defined
through the following random composition%
\begin{equation}
\mathfrak{F}_{t}^{\alpha ,\beta }=\mathfrak{A}_{\mathfrak{L}_{t}^{\beta
}}^{\alpha },\quad t>0 \label{time-change-Sec2}
\end{equation}
independent from $N$, where $\mathfrak{A}_{t}^{\alpha}$, $t>0$ is an $\alpha$-stable subordinator and $\mathfrak{L}_{t}^{\beta }=\inf\{s \geq 0\,: \,\mathfrak{A}_{s}^{\beta}>t \}$, $t>0$ is the inverse to a stable subordinator of order $\beta \in (0,1)$, all independent from each other.

The process \eqref{dr} can resemble the compensated Poisson process, defined as $N(t)-\lambda t$ (where $\lambda$ is the parameter of $N(t)$) see e.g. \cite{appB}. We want to remark here that the two processes are completely different since \eqref{dr} is, for any $\alpha, \beta \in (0,1]$, a non decreasing process.

We further generalize (\ref{dr}) by considering a "fractional version" of
it, obtained by a random time-change of $N$, i.e. as%
\begin{equation}
N(\mathfrak{F}_{t}^{\gamma ,\beta })+a\mathfrak{F}_{t}^{\alpha ,\beta
},\quad t>0,\;a\geq 0,\;\alpha ,\gamma ,\beta \in (0,1].  \label{dr3}
\end{equation}%
As particular cases of (\ref{dr3}), when the drift coefficient is equal to
zero, we can derive two fractional Poisson processes already studied in the
literature. For $a=0$ and $\gamma =1$, it reduces to the time-fractional
Poisson process $N_{\beta }(t),$ $t>0$\ which is studied in \cite{Rep}, \cite%
{Mai}, \cite{BegOrs}: indeed it is proved in \cite{Mee} that $N_{\beta }$
coincides with $N(\mathfrak{L}_{t}^{\beta })$. On the other hand, for $a=0$ and $\beta =1$, the process (\ref{dr3}) reduces to the space-fractional Poisson process studied in \cite{ors-pol-SPL},
which can be defined as $N(\mathfrak{A}_{t}^{\gamma })$. Subordinated Poisson semigroups have been also investigated in \cite{forst81} where measures of the form
\begin{equation*}
\int_0^\infty P_t \mu(dt) = \sum_{k=0}^\infty a_k \delta_k
\end{equation*}
($\delta_k$ is the Dirac measure at $k$) with
\begin{equation*}
P_t = e^{-t} \sum_{k=0}^\infty \frac{t^k}{k!} \delta_k
\end{equation*}
have been characterized in terms of the properties of the sequence $\{a_k\}_{k \in \mathbb{N}_0}$.

Throughout the paper, we are interested in studying the fractional
differential equation satisfied by the density of the processes defined
above. These equations will be expressed in terms of the translation operator defined as%
\begin{equation}
e^{\vartheta \partial _{x}}f(x)=f(x+\vartheta )  \label{translatOp}
\end{equation}%
for $x,\vartheta \in \mathbb{R}$ and an analytic function $f:\mathbb{R}%
\mapsto \mathbb{R}$. The rule \eqref{translatOp} can be formally obtained by
considering the Taylor expansion of $f$ near $x$ written as
\begin{equation*}
f(\vartheta )=\sum_{k=0}^{\infty }\frac{(\vartheta -x)^{k}}{k!}f^{(k)}(x)
\end{equation*}%
and therefore
\begin{equation*}
f(x+\vartheta )=\sum_{k=0}^{\infty }\frac{\vartheta ^{k}}{k!}\partial
_{x}^{k}f(x)=\sum_{k=0}^{\infty }\frac{(\vartheta \partial _{x})^{k}}{k!}%
f(x)=e^{\vartheta \partial _{x}}f(x).
\end{equation*}%
The Taylor series can be considered also for the class of bounded continuous
functions on $(0,+\infty )$ (see for example \cite{Fel71, HillePhil57}), so
that we extend the rule \eqref{translatOp} to such class of functions.

Let $N(t)$, $t>0$ be a Poisson process with rate $\lambda >0$; we write its
distribution as follows%
\begin{equation*}
p_{k}(t)=\frac{(-\lambda \partial _{\lambda })^{k}}{k!}e^{-\lambda t},
\end{equation*}%
which solves the differential equation
\begin{align}
\frac{\partial p_{k}}{\partial t}=& -\lambda \left( I-B\right) p_{k}(t)
\label{eqpois} \\
=& -\lambda \big(p_{k}(t)-p_{k-1}(t)\big),\quad k\in \mathbb{N}_{0},\;t>0
\notag
\end{align}%
with
\begin{equation*}
p_{k}(0)=\left\{
\begin{array}{ll}
1, & k=0 \\
0, & k\leq 1%
\end{array}%
\right. .
\end{equation*}%

We denote by $B$ the (discrete) backward difference operator, i.e. $Bu(x)=u(x-1)$ for all integers $x$. By means of (\ref{translatOp}) we can
also rewrite the well-known probability generating function of $N(t)$, $t>0,$
as follows
\begin{equation*}
\mathbb{E}\,u^{N(t)}=\sum_{k=0}^{\infty }\frac{(-\lambda u\partial _{\lambda
})^{k}}{k!}e^{-\lambda t}=e^{-\lambda u\partial _{\lambda }}e^{-\lambda
t}=e^{-\lambda t(1-u)}.
\end{equation*}%
Let us denote the waiting time of the $k$-th event for $N$, as%
\begin{equation}
T_{k}=\inf \{t\geq 0\,:\,N(t)>k\}.  \label{wai}
\end{equation}%
We will study the analogues of (\ref{wai}) for the processes (\ref{dr}) and (%
\ref{dr3}) and obtain their governing equations.

Further in the paper we will consider the solution to the
Poisson driven stochastic differential equation, for well-defined functions $b$ and $f$ (see for example \cite{Traple96})
\begin{equation}
dY_{t}=b(Y_t)dt+f(Y_t)dN_{t}  \label{SDEPois}
\end{equation}%
with $\mathbb{E}dN_{t}=\lambda dt$ and
\begin{equation}
dN_{t}=\left\{
\begin{array}{ll}
1, & \text{ Poisson arrival at time } t, \\
0, & \text{ elsewhere}.%
\end{array}%
\right.  \label{dN}
\end{equation}%
Here, the function $f$ plays the role of jump function. The partial differential equation corresponding to \eqref{SDEPois} is a
transport equation of the form
\begin{equation}
\frac{\partial u}{\partial t}=-\frac{\partial }{\partial x}\Big(b(x)u\Big)%
-\lambda (I-K)u  \label{transpNeq}
\end{equation}%
where $I$ is the identity operator and $K$ is the Frobenius-Perron operator
associated with the transformation $x\mapsto x-f(x)$. If $f\neq 1$, then we
have a generalized jump which equals $f$ at each Poisson arrival as equation %
\eqref{SDEPois} entails.

In Section \ref{Sec2} we consider a L\'{e}vy process time-changed with a
Poisson process with deterministic drift, i.e.%
\begin{equation}
X(N(t)+at),  \label{lev}
\end{equation}%
where $X(t),$ $t>0$ is a L\'{e}vy process independent from $N.$ Indeed the
composition (\ref{lev}) is meaningful since the process representing the time is positive and real-valued. Section \ref{Sec3} is devoted to the analysis of the Poisson process time-changed by the process \eqref{time-change-Sec2} : we find connections with fractional and higher-order equations and derive explicit representations for the density of the hitting time. In Section \ref{SecPoisDrift} we study the time-changed Poisson process with random drift \eqref{dr3} and its hitting time. Finally in Section  \ref{lastSect} we consider more general versions of \eqref{lev} where $X$ is time-changed by the process \eqref{dr3} and also by the (independent) hitting time process of \eqref{dr3} with $\beta=1$. In both cases we derive the governing equations.

\section{Poisson process with drift}
\label{Sec2}
In order to consider the Poisson process with continuous drift we introduce
the shift operator, which we define as%
\begin{equation*}
Ku(x)=\left\{
\begin{array}{ll}
e^{-\partial _{x}}u(x), & \text{ if }x\in \mathbb{R}_{+} \setminus \mathbb{Z}_{+} \\
Bu(x), & \text{ if }x\in \mathbb{Z}_{+}%
\end{array}%
\right.
\end{equation*}%
where $B$ is the backward difference operator and $e^{-\partial _{x}}$  is defined
in \eqref{translatOp}.

\begin{te}
The process
\begin{equation}
N(t)+at,\quad t>0,\;a>0  \label{PoisDrift}
\end{equation}%
has probability law
\begin{equation}
p_{x}(t)=e^{-\lambda t}\sum_{k=0}^{\infty }\frac{(\lambda t)^{k}}{k!}\delta
(x-k-at),\quad x\geq at,\;a>0,\; t>0,  \label{PoisDriftLaw}
\end{equation}
which is the solution to the transport equation
\begin{align}
\left( \frac{\partial }{\partial t}+a\frac{\partial }{\partial x}\right)
p_{x}(t)=& -\lambda \left( I-K\right) p_{x}(t)  \label{eqpoisa} \\
=& -\lambda \big(p_{x}(t)-p_{x-1}(t)\big)  \notag
\end{align}%
with initial and boundary conditions
\begin{equation*}
\left\{
\begin{array}{l}
p_{x}(0)=\delta (x) \\
p_{0}(t)=\delta (at)e^{-\lambda t}%
\end{array}%
\right.
\end{equation*}%
where $\delta $ is the Dirac delta function.
\label{theo-green-pois}
\end{te}

\begin{proof}
The Laplace transform of \eqref{PoisDrift} is given by
\begin{equation}
\mathbb{E} e^{-\xi N(t) -\xi a t} = e^{- \xi a t} \mathbb{E} e^{-\xi N(t)} = \exp \left( -\xi a t - \lambda t (1-e^{-\xi})\right). \label{laplaw1}
\end{equation}
We can prove that \eqref{PoisDriftLaw} is the law of \eqref{PoisDrift} by checking that
\begin{align*}
\widetilde{p_\xi}(t) = & \int_0^\infty e^{-\xi x} p_x(t)dx \\
= & e^{-\xi a t - \lambda t} \sum_{k=0}^{\infty} \frac{(\lambda t)^k}{k!} e^{-\xi k} \\
= & \exp \left( -\xi a t - \lambda t + \lambda t e^{-\xi}\right)
\end{align*}
coincides with \eqref{laplaw1}. We observe that
\begin{align*}
\int_0^\infty e^{-\xi x} \left(I -K\right) p_x(t)  dx= & \int_0^\infty e^{-\xi x} \big( p_x(t)  - p_{x-1}(t) \big)dx =  (1- e^{-\xi}) \widetilde{p_{\xi}}( t),
\end{align*}
so that the Laplace transform of equation \eqref{eqpoisa} takes the form
\begin{equation*}
\frac{\partial  \widetilde{p_{\xi}}}{\partial t}(t) = \left( - a \xi  - \lambda (1- e^{-\xi})\right) \widetilde{p_{\xi}}(t).
\end{equation*}
We immediately get
\begin{equation*}
 \widetilde{p_{\xi}}( t) = \exp\left( - a \xi t - \lambda t (1- e^{-\xi}) \right)
\end{equation*}
since $ \widetilde{p_{\xi}}(0)=1$.
\end{proof}

\begin{os}
In Theorem \ref{theo-green-pois} we have considered a series representation involving the generalized delta function.  Let us consider an absolutely integrable function $f$ with compact support in the positive real line. We notice that
\begin{equation}
P_t f(x) = \mathbb{E}f(x-N(t)-at) = e^{-\lambda t}\sum_{k=0}^{\infty }\frac{(\lambda t)^{k}}{k!}f
(x-k-at) \label{Pois-semigroup}
\end{equation}
is the transition semigroup associated to the process \eqref{PoisDrift} with initial datum $f \in L^1(\mathbb{R_+})$. Furthermore, we get that
\begin{equation*}
\widetilde{P_t f}(\xi) = \widetilde{f}(\xi)\, \widetilde{p_{\xi}}( t).
\end{equation*}
\label{remark-semigroup-pois}
\end{os}

The homogeneous Poisson process is one of the most well-known L\'{e}vy processes. Let us consider the one-dimensional L\'{e}vy process $X(t)$, $t>0$ with L\'{e}vy symbol
\begin{equation}
\label{Levy-symb-X}
\Psi (\xi ) = ib\xi + \frac{1}{2}\xi^2 - \int_{\mathbb{R}- \{0\}} (e^{i\xi y} -1 -i \xi y \mathbf{1}_{(|y|\leq 1)}) M(dy)
\end{equation}
($M$ is the so called L\'{e}vy measure and $b \in \mathbb{R}$ is a drift coefficient) and infinitesimal generator
\begin{equation}
\label{pseudo-diff-A}
\mathcal{A}f(x)=-\frac{1}{2\pi}\int_{\mathbb{R}}e^{-i\xi x}\Psi (\xi )\widehat{f}(\xi )d\xi
\end{equation}%
for all functions in the domain
\begin{equation}
\label{pseudo-diff-domA}
D(\mathcal{A}) = \left\lbrace f \in L^2(\mathbb{R},dx)\,:\, \int_{\mathbb{R}} \Psi(\xi) |\widehat{f}(\xi)|^2 d \xi< \infty \right\rbrace
\end{equation}
($\widehat{f}$ is the Fourier transform of $f$). Therefore, $-\Psi $
is the Fourier multiplier of $\mathcal{A}$ and $\mathbb{E}\exp i\xi X(t) = \exp -t \Psi(\xi)$.  We recall that $M$ is a Borel measure on $\mathbb{R}^d - \{0\}$ such that
\begin{equation*}
\int (y^2 \wedge 1)M(dy) < \infty \quad \textrm{ or equivalently } \quad \int \frac{y^2}{1+ y^2}M(dy) < \infty
\end{equation*}
where $a \wedge b = \min\{a,b\}$. If $\mathfrak{D}_t$, $t>0$ is a non-decreasing L\'{e}vy process, that is a subordinator (not necessarily stable), then its L\'{e}vy symbol is written as
\begin{equation}
\psi(\xi) = ib\xi + \int_0^\infty \left( e^{i\xi y} -1 \right)M(dy) \label{symb-sub-eta}
\end{equation}
where $b\geq 0$ and the L\'{e}vy measure $M$ satisfies the following requirements: $M(-\infty, 0)=0$ and
\begin{equation}
\int (y \wedge 1)M(dy) < \infty \quad \textrm{ or equivalently } \quad \int \frac{y}{1+ y}M(dy) < \infty. \label{levy-meas-cond2}
\end{equation}
Thus, we get that $\mathbb{E}\exp \big(-\xi \mathfrak{D}_t \big)= \exp \big( t \psi(i\xi)\big)$ and $-\psi(i\xi)$ is the Laplace exponent of $\mathfrak{D}_t$, $t>0$.

Let $P_t=e^{t\mathcal{A}}$ be the semigroup of $X(t)$. Then, $P_t$ is a Feller semigroup (invariant in $C_\infty$ and a strongly continuous contraction semigroup on the Banach space $(C_\infty, |\cdot |_\infty)$ of the infinitely differentiable functions under the sup-norm).  In particular, we are able to compute the semigroup and its generator as pseudo-differential operators (as in formulas \eqref{pseudo-diff-A} and \eqref{pseudo-diff-domA}) and we say that $\widehat{P_t} = e^{-t\Psi}$ is the symbol of $P_t$.

We now focus on the time-changed process
\begin{equation}
X(N(t)+at),\quad t\geq 0,\;a\geq 0  \label{sub-proc-X}
\end{equation}
involving a continuous time-change with Poissonian jumps and such that $X(0)=0$.
\begin{lem}
\label{lemmaXN}
The infinitesimal generator of \eqref{sub-proc-X} is
\begin{equation}
\mathcal{L}f(x)=a\mathcal{A}f(x)-\lambda \int_{\mathbb{R}}\left(
f(x+y)-f(x)\right) F_{X}(dy) \label{inf-gen-XN}
\end{equation}%
where $F_{X}(dy)=f_X(y)dy$ and $f_X$ is the density law of $X(1)$.
\end{lem}
\begin{proof}
We get that
\begin{align}
\mathbb{E}e^{i\xi \left( X(N(t)+at)\right) }=& \mathbb{E}e^{-\left(
N(t)+at\right) \Psi (\xi )}  =  e^{-at\Psi (\xi )}\mathbb{E}e^{-\Psi (\xi )N(t)}  \notag \\
=& \exp \left( -at\Psi (\xi )-\lambda t\left( 1-e^{-\Psi (\xi )}\right)
\right)  \notag \\
=& \exp \left( -t\Phi (\xi )\right)  \label{Phi-symbol}
\end{align}%
and therefore,
\begin{equation*}
\mathcal{L}f(x)=- \frac{1}{2\pi}\int_{\mathbb{R}}e^{-i\xi x}\Phi (\xi )\widehat{f}(\xi )d\xi
\end{equation*}%
is the infinitesimal generator of \eqref{sub-proc-X}. Indeed, we immediately see that $-\Phi$ is the Fourier multiplier of \eqref{inf-gen-XN}.
\end{proof}
We notice that $F_X$ in \eqref{inf-gen-XN} is a non singular measure. Thus,
\begin{align*}
\int_{\mathbb{R}}\left( f(x+y)-f(x)\right) F_{X}(dy) = & \int_{\mathbb{R}}f(x+y) F_{X}(dy) - f(x)\\
= & e^{\mathcal{A}}f(x) - f(x) \\
= & (P_1 - P_0) f(x).
\end{align*}
In  Section \ref{lastSect} we will extend this result to the case where the time-change in \eqref{sub-proc-X} is represented by the process \eqref{dr3}.

\section{Time-changed Poisson process}
\label{Sec3}

We begin our analysis by studying the following composition
\begin{equation}
\mathfrak{F}_{t}^{\alpha ,\beta }=\mathfrak{A}_{\mathfrak{L}_{t}^{\beta
}}^{\alpha },\quad t>0  \label{Fproc}
\end{equation}%
where $\mathfrak{L}_{t}^{\beta }$, $t>0$ is the inverse of the stable
subordinator $\mathfrak{A}_{t}^{\beta }$, $t>0$. The stable process $%
\mathfrak{A}_{t}^{\alpha }$, $t>0$ is a L\'{e}vy process with non-negative
increments and therefore non-decreasing paths. Therefore, the inverse to a
stable subordinator $\mathfrak{L}_{t}^{\alpha }$, $t>0$ can be regarded as a
hitting time. Indeed, we define the inverse process by writing
\begin{equation}
Pr\{\mathfrak{L}_{t}^{\alpha }<x\}=Pr\{\mathfrak{A}_{x}^{\alpha }>t\}
\label{relP}
\end{equation}%
which means that
\begin{equation*}
\mathfrak{L}_{t}^{\alpha }=\inf \{s\geq 0\,:\,\mathfrak{A}_{s}^{\alpha
}\notin (0,t)\}.
\end{equation*}%
From the fact that
\begin{equation}
\mathbb{E}e^{-\xi \mathfrak{A}_{t}^{\alpha }}=e^{-t\xi ^{\alpha }}
\label{lap-sub}
\end{equation}%
after some algebra, formula \eqref{relP} says that
\begin{equation}
\mathbb{E}e^{-\xi \mathfrak{L}_{t}^{\alpha }}=E_{\alpha }(-t^{\alpha }\xi ),
\label{lap-inv}
\end{equation}%
where $E_{\rho}$ is a special case (for $\varrho =1$) of the generalized
Mittag-Leffler function%
\begin{equation*}
E_{\rho ,\varrho}(z)=\sum_{k=0}^{\infty }\frac{z^{k}}{\,\Gamma (\rho
k+\varrho )},\quad \Re \{\rho \}>0,\;\rho ,\varrho ,z\in \mathbb{C}.
\end{equation*}

The density of the inverse process $\mathfrak{L}_{t}^{\alpha }$, $t>0$, can
be written in terms of the Wright function
\begin{equation*}
W_{\rho ,\varrho }(z)=\sum_{k=0}^{\infty }\frac{z^{k}}{k!\,\Gamma (\rho
k+\varrho )},\quad \Re \{\varrho \}>0,\;\rho >-1,\;z\in \lbrack 0,\infty )
\end{equation*}%
as follows
\begin{equation}
l_{\alpha }(x,t)=\frac{1}{t^{\alpha }}W_{\alpha ,1-\alpha }\left( -\frac{x}{%
t^{\alpha }}\right) ,\quad x\geq 0,\;t>0 \label{density-l}
\end{equation}%
whereas, for the density of $\mathfrak{A}_{t}^{\alpha }$, $t>0$ we can write
(\cite{Dov-Wright})
\begin{equation*}
h_{\alpha }(x,t)=\frac{\alpha t}{x}l_{\alpha }(t,x).
\end{equation*}

Let
\begin{equation}
f_{t}^{\alpha ,\beta }(x)=\int_{0}^{\infty }h_{\alpha }(x,s)l_{\beta
}(s,t)ds,\quad x\geq 0,\;t>0,\;\alpha ,\beta \in (0,1)  \label{lawFproc}
\end{equation}%
be the law of the process $\mathfrak{F}_{t}^{\alpha ,\beta }$, $t>0$. Then
it is easy to check that the governing equation of \eqref{lawFproc} is given
by%
\begin{equation}
\left( \mathcal{D}^\beta_t -\partial _{x}^{\alpha
}\right) f_{t}^{\alpha ,\beta }(x)=0,\quad x\geq 0,\;t>0  \label{pdeFproc}
\end{equation}%
subject to the initial and boundary conditions
\begin{equation*}
\left\{
\begin{array}{l}
f_{0}^{\alpha ,\beta }(x)=\delta (x), \\
f_{t}^{\alpha ,\beta }(0)=0,%
\end{array}%
\right.
\end{equation*}%
where
\begin{equation*}
\mathcal{D}^\beta_t \, u=\partial _{t}^{\beta
}u-u(0^{+})\frac{t^{-\beta }}{\Gamma (1-\beta )}
\end{equation*}%
is the Dzhrbashyan-Caputo fractional derivative and
\begin{equation*}
\partial _{z}^{\alpha }u=\frac{\partial^\alpha u}{\partial z^\alpha} = \frac{1}{\Gamma (1-\alpha )}\frac{\partial }{%
\partial z}\int_{0}^{z}\frac{u(s)\,ds}{(z-s)^{\alpha }}
\end{equation*}%
is the Riemann-Liouville fractional derivative.

The following result will turn out to be useful further in the text.

\begin{lem}
The Laplace transform of \eqref{lawFproc} is given by
\begin{equation}
\mathbb{E}e^{-\xi \mathfrak{F}_{t}^{\alpha ,\beta }}=E_{\beta }(-t^{\beta
}\xi ^{\alpha })  \label{LapF}
\end{equation}%
and satisfies the following equation
\begin{equation}
\frac{\partial }{\partial t}\mathbb{E}e^{-\xi \mathfrak{F}_{t}^{\alpha
,\beta }}=\frac{\beta \xi }{\alpha t}\frac{\partial }{\partial \xi }\mathbb{E%
}e^{-\xi \mathfrak{F}_{t}^{\alpha ,\beta }},\quad \xi ,t>0.  \label{dertF}
\end{equation}
\end{lem}

\begin{proof}
Since $\mathfrak{F}^{\alpha, \beta}_t$ has non-negative increments, the Laplace transform exists, i.e.
$$\mathbb{E} e^{-\xi \mathfrak{F}^{\alpha, \beta}_t} < \infty$$
and can be easily written as in \eqref{LapF}. In order to check \eqref{dertF} we recall that
\begin{align*}
\frac{d}{d z} E_\beta(-z)= & \frac{d}{d z} \sum_{k=0}^{\infty} \frac{(-z)^k}{\Gamma(\beta k + 1)} = - \frac{1}{\beta}  E_{\beta, \beta}(-z),
\end{align*}
so that
\begin{align}
t \frac{\partial}{\partial t} E_\beta(-t^\beta \xi^\alpha) = & - \xi^\alpha t^{\beta}  E_{\beta, \beta}(-t^\beta \xi^\alpha) \label{Ftmp1}
\end{align}
and
\begin{align}
\xi \frac{\partial}{\partial \xi} E_\beta(-t^\beta \xi^\alpha) = & - \frac{\alpha}{\beta} \xi^{\alpha} t^\beta E_{\beta, \beta}(-t^\beta \xi^\alpha). \label{Ftmp2}
\end{align}
Thus, by considering \eqref{Ftmp1} and \eqref{Ftmp2} together, we obtain
\begin{equation*}
\alpha t \frac{\partial}{\partial t} E_\beta(-t^\beta \xi^\alpha)  = \beta \xi \frac{\partial}{\partial \xi} E_\beta(-t^\beta \xi^\alpha)
\end{equation*}
which coincides with \eqref{dertF}.
\end{proof}

Let us consider the exit time of \eqref{Fproc} from the interval $(0,t)$,
i.e.%
\begin{equation*}
\mathcal{T}_{t}^{\alpha ,\beta }=\inf \{s\geq 0\,:\,\mathfrak{F}_{s}^{\alpha
,\beta }\notin (0,t)\},\quad t>0.
\end{equation*}%
Since $\mathfrak{F}_{t}^{\alpha ,\beta }$ has non-decreasing paths, we can
argue that it satisfies the relation
\begin{equation}
Pr\{\mathcal{T}_{t}^{\alpha ,\beta }<x\}=Pr\{\mathfrak{F}_{x}^{\alpha ,\beta
}>t\}.  \label{cond3}
\end{equation}

\begin{lem}
For $\alpha ,\beta \in (0,1)$, the following result holds true
\begin{equation*}
\mathcal{T}_{t}^{\alpha ,\beta }\overset{law}{=}\mathfrak{F}_{t}^{\beta
,\alpha }
\end{equation*}%
and thus
\begin{equation}
\mathfrak{F}_{t}^{\beta ,\alpha }\overset{law}{=}\inf \left\{ s\geq 0\,:\,%
\mathfrak{F}_{s}^{\alpha ,\beta }\notin (0,t)\right\} ,  \label{Fhitting}
\end{equation}%
where $\overset{law}{=}$ denotes the equality of the finite
dimensional distributions.
\end{lem}

\begin{proof}
Considering together \eqref{lap-sub} and
\begin{equation*}
\int_0^\infty e^{-\mu t} \frac{Pr\{\mathfrak{L}^\alpha_t \in dx \}}{dx} dt = \mu^{\alpha -1}e^{-x \mu^\alpha},
\end{equation*}
the composition
\begin{equation*}
\mathfrak{F}^{\beta, \alpha}_t = \mathfrak{A}^\beta_{\mathfrak{L}^\alpha_t}, \quad t>0
\end{equation*}
has double Laplace transform given by
\begin{equation}
\int_0^\infty e^{-\mu t} \, \mathbb{E} e^{- \xi \mathfrak{F}^{\beta, \alpha}_t} dt = \frac{\mu^{\alpha -1}}{\mu^\alpha + \xi^\beta}.\label{dobleL1}
\end{equation}
We now assume that \eqref{Fhitting} holds. By applying \eqref{cond3}, we get that
\begin{equation}
\frac{Pr\{ \mathfrak{F}^{\beta, \alpha}_t \in dx \}}{dx} =  - \frac{\partial}{\partial x} Pr\{ \mathfrak{F}^{\alpha, \beta}_x < t \}. \label{lapFlaw}
\end{equation}
The Laplace transform of \eqref{lapFlaw} reads
\begin{align*}
\int_0^\infty e^{-\mu t} \frac{Pr\{ \mathfrak{F}^{\beta, \alpha}_t \in dx \}}{dx} dt=  & - \frac{\partial}{\partial x} \frac{1}{\mu} \int_0^\infty e^{-\mu t} Pr\{ \mathfrak{F}^{\alpha, \beta}_x \in dt \}\\
= &  - \frac{\partial}{\partial x} \frac{1}{\mu} \mathbb{E} e^{-\mu  \mathfrak{F}^{\alpha, \beta}_x } =   - \frac{\partial}{\partial x} \frac{1}{\mu} E_\beta(- x^\beta \mu^\alpha)\\
= & [\textrm{by } \eqref{Ftmp1}] =  x^{\beta -1}\mu^{\alpha -1} E_{\beta , \beta}(-x^\beta \mu^\alpha),
\end{align*}
while its double Laplace transform is given by
\begin{align}
\int_0^\infty e^{-\xi x} \int_0^\infty e^{-\mu t} Pr\{ \mathfrak{F}^{\beta, \alpha}_t \in dx \}dt = & \int_0^\infty e^{-\xi x} x^{\beta -1}\mu^{\alpha -1} E_{\beta , \beta}(-x^\beta \mu^\alpha)dx\notag \\
= & \frac{\mu^{\alpha -1}}{\mu^\alpha + \xi^\beta}.\label{dobleL2}
\end{align}
Formula \eqref{dobleL2} coincides with \eqref{dobleL1} and this proves the claim \eqref{Fhitting}.
\end{proof}

\begin{os}
We observe that
\begin{equation}
\frac{\mathfrak{F}_{t}^{\alpha ,\beta }}{t}\xrightarrow{t \to \infty}%
\left\{
\begin{array}{ll}
+\infty  & \beta >\alpha  \\
W & \beta =\alpha  \\
0 & \beta <\alpha
\end{array}%
\right.
\label{as}
\end{equation}%
where $W$ represents the ratio of two independent stable subordinators of order $\alpha $ and possesses Lamperti distribution (see for example \cite{lanc2010}). The convergence in \eqref{as} must be understood in distribution as we can immediately verify by looking at
\begin{equation*}
\mathbb{E}\exp \left( -\mu \frac{\mathfrak{F}_{t}^{\alpha ,\beta }}{t}%
\right) =E_{\beta }(-\mu ^{\alpha }t^{\beta -\alpha }).
\end{equation*}%
Indeed, it is easy to check that $E_{\beta }(-\infty )=0$ and $E_{\beta }(0)=1$. Thus, for $%
t\rightarrow \infty $, we have that
\begin{equation*}
Pr\{\mathfrak{F}_{t}^{\alpha ,\beta }>t\}=1,\quad \textrm{if }\;  \beta >\alpha ,
\end{equation*}%
\begin{equation*}
Pr\{\mathfrak{F}_{t}^{\alpha ,\beta }<t\}=1,\quad \textrm{if }\; \beta <\alpha .
\end{equation*}
\end{os}

We study now the time-changed Poisson process. We refer to $N(t)$, $t>0$ as
the base process and to $\mathfrak{F}_{t}^{\alpha ,\beta }$ as the
time-change process. Thus, the resulting composition
\begin{equation}
N(\mathfrak{F}_{t}^{\alpha ,\beta }),\quad t>0, \quad \alpha, \beta \in (0,1]  \label{Nsub}
\end{equation}%
is a time-changed Poisson process with probability law
\begin{equation*}
Pr\{N(\mathfrak{F}_{t}^{\alpha ,\beta })=k\}=p_{k}(t;\alpha ,\beta
)=\int_{0}^{\infty }p_{k}(s)f_{t}^{\alpha ,\beta }(s)ds,
\end{equation*}%
where we use the following notation
\begin{align}
p_{k}(t;\alpha ,\beta )=& \mathbb{E}\frac{(-\lambda \partial _{\lambda })^{k}
}{k!}e^{-\lambda \mathfrak{F}_{t}^{\alpha ,\beta }}  =  \frac{(-\lambda \partial _{\lambda })^{k}}{k!}\mathbb{E}e^{-\lambda \mathfrak{F}_{t}^{\alpha ,\beta }} =  \frac{(-\lambda \partial _{\lambda  })^{k}}{k!}E_{\beta }(-t^{\beta}\lambda ^{\alpha }).  \label{probNsub}
\end{align}%
The probability generating function of $N(\mathfrak{F}_{t}^{\alpha ,\beta })$%
, $t>0$, is given by
\begin{align}
G_{\beta }^{\alpha }(u,t)=& \mathbb{E}u^{N(\mathfrak{F}_{t}^{\alpha ,\beta
})}  =  \sum_{k=0}^{\infty }\frac{(-u\lambda \partial _{\lambda })^{k}}{k!}%
\mathbb{E}e^{-\lambda \mathfrak{F}_{t}^{\alpha ,\beta }}  \label{gen-fun-Poi-s-t} \\
=& e^{-u\lambda \partial _{\lambda }}\mathbb{E}e^{-\lambda \mathfrak{F}%
_{t}^{\alpha ,\beta }} =  \mathbb{E}e^{-\lambda (1-u)\mathfrak{F}_{t}^{\alpha ,\beta }}  \notag \\
=& E_{\beta }(-t^{\beta }\lambda ^{\alpha }(1-u)^{\alpha }).  \notag
\end{align}

\begin{te}
The probability law \eqref{probNsub} of the time-changed Poisson process \eqref{Nsub} is the solution to the fractional differential equation
\begin{equation}
\left( \mathcal{D}^\beta_t +\lambda ^{\alpha }(I-B)^{\alpha
}\right) p_{k}(t;\alpha ,\beta )=0,\quad k=0,1,2\ldots ,\;t>0
\label{eq-Pois-frac-time-space}
\end{equation}%
with initial condition
\begin{equation*}
p_{k}(0;\alpha ,\beta )=\left\{
\begin{array}{ll}
0, & k\geq 1, \\
1, & k=0,%
\end{array}%
\right.
\end{equation*}%
where
\begin{equation*}
(I-B)^{\alpha }=\sum_{j=0}^{\infty }(-1)^{j}\binom{\alpha }{j}B^{j}.
\end{equation*}%
Furthermore, the waiting time of the k-th event of $N(\mathfrak{F}%
_{t}^{\alpha ,\beta })$, $t>0$, i.e.
\begin{equation*}
T_{k}^{\alpha ,\beta }=\inf \left\{ s\geq 0\,:\,N(\mathfrak{F}_{s}^{\alpha
,\beta })>k\right\}
\end{equation*}%
has density given by
\begin{equation*}
Pr\{T_{k}^{\alpha ,\beta }\in dt\} / dt=\frac{\beta k}{\alpha t}\frac{%
(-\lambda \partial _{\lambda })^{k}}{k!}E_{\beta }(-t^{\beta }\lambda
^{\alpha })=\frac{\beta k}{\alpha t}p_{k}(t;\alpha ,\beta )
\end{equation*}%
and
\begin{equation*}
T_{k}^{\alpha ,\beta }\overset{law}{=}\mathfrak{F}_{T_{k}}^{\beta ,\alpha
},\quad k=1,2,\ldots .
\end{equation*}
\end{te}

\begin{proof}
From the probability generating function \eqref{gen-fun-Poi-s-t} we get that
\begin{equation*}
\mathcal{D}^\beta_t G^\alpha_\beta(u, t) = - \lambda^\alpha (1-u)^\alpha\, G^\alpha_\beta(u, t),
\end{equation*}
since the Mittag-Leffler is an eigenfunction for the Dzhrbashyan-Caputo fractional derivative. Let us consider the auxiliary function $f \in L^1(\mathbb{R}_+)$.  By considering the Bernstein function
\begin{equation}
x^\alpha = \frac{\alpha}{\Gamma(1-\alpha)} \int_{0}^\infty \left( 1-e^{-sx}\right) \frac{ds}{s^{\alpha +1}}
\label{typical-bern-function}
\end{equation}
we formally write
\begin{align*}
(I-B)^\alpha f = & \frac{\alpha}{\Gamma(1-\alpha)} \int_{0}^\infty \left( f - e^{-s(I-B)} f \right) \frac{ds}{s^{\alpha +1}}  \\
= & \frac{\alpha}{\Gamma(1-\alpha)} \int_{0}^\infty \left( f - P_s f \right) \frac{ds}{s^{\alpha +1}}
\end{align*}
where we denote by $P_s$ the transition semigroup
\begin{equation}
P_sf(x) = \mathbb{E}f(x- N(s)) = \sum_{k=0}^\infty f(x-k)  \frac{s^k}{k!}e^{- s}
\label{ps}
\end{equation}
as in Remark \ref{remark-semigroup-pois} with $a=0$ and $\lambda=1$.  The Laplace transform of \eqref{ps} is given by
\begin{equation*}
\widetilde{P_sf}(\xi) = \widetilde{f}(\xi)\, e^{- s \Psi(i \xi)} = \widetilde{f}(\xi)\, \mathbb{E} e^{-\xi N(s)},
\end{equation*}
where $\Psi(\xi)= (1- e^{i\xi})$ and $N$ is now a Poisson process with $\lambda=1$.
For $ 0 < u < 1$, we immediately get
\begin{equation*}
\widetilde{P_sf}(-\log u) = \widetilde{f}(-\log u) e^{-s(1-u)} = \widetilde{f}(-\log u) \, \mathbb{E} u^{N(s)}
\end{equation*}
and therefore, we have that
\begin{align*}
\widetilde{(I-B)^\alpha f}(-\log u) = & \frac{\alpha}{\Gamma(1-\alpha)} \int_{0}^\infty \left( \widetilde{f}(-\log u) - \widetilde{P_sf}(-\log u) \right) \frac{ds}{s^{\alpha +1}}\\
= & \frac{\alpha}{\Gamma(1-\alpha)} \int_{0}^\infty \left( 1 - e^{-s (1-u)}  \right) \frac{ds}{s^{\alpha +1}}\; \widetilde{f}(-\log u)\\
= & (1-u)^\alpha\, \widetilde{f}(-\log u).
\end{align*}
We now consider a function $f(x,t)$ such that $|\frac{\partial}{\partial t}f(x,t)| \leq t^{\gamma -1} g(x)$ with $\gamma>0$ and $g \in L^\infty(\mathbb{R}_+)$ which is, as a function of $x$, consistent with the previous assumption. Thus, by imposing that
\begin{equation*}
-\lambda^\alpha \widetilde{(I-B)^\alpha f}(-\log u, t) = - \lambda^\alpha\,  (1-u)^\alpha\, \widetilde{f}(-\log u, t) = \mathcal{D}^\beta_t\, \widetilde{f}(-\log u, t)
\end{equation*}
 and considering that $G^\alpha_\beta(u,0)=1$, we get
\begin{equation*}
\widetilde{f}(-\log u, t) = G^\alpha_\beta (u, t),
\end{equation*}
which proves that equation \eqref{eq-Pois-frac-time-space} is satisfied. The fact that $|\frac{\partial}{\partial t} f(\cdot, t)| \leq t^{\gamma-1}$ for some $\gamma >0$ is a standard requirement for the existence of $\mathcal{D}^\beta_t$ which comes directly from the definition of the fractional derivative. Furthermore, the Laplace techniques ensure uniqueness.\\

Now we focus on the waiting time $T^{\alpha, \beta}_k$, $k \in \mathbb{N}$: its probability distribution function can be written as
\begin{align*}
Pr\{ T^{\alpha, \beta}_k \leq t \} = & Pr\{ N(\mathfrak{F}^{\alpha, \beta}_t) \geq k \}\\
= & \sum_{m=k}^\infty p_m(t; \alpha, \beta)\\
= & \sum_{m=k}^\infty \frac{(-\lambda)^m}{m!} \partial_\lambda^m \lambda \frac{\mathbb{E} e^{-\lambda \mathfrak{F}^{\alpha, \beta}_t}}{\lambda}.
\end{align*}
Since
\begin{align*}
\partial_{\lambda}^m \lambda f(\lambda ) = & \partial_{\lambda}^{m-1} \left[ 1 + \lambda \partial_{\lambda } \right] f(\lambda )\\
= & \partial_{\lambda}^{m-2} \left[ 2\partial_{\lambda} + \lambda \partial_{\lambda }^2 \right] f(\lambda )\\
= & [\text{for any } k \leq m]\\
= & \partial_{\lambda}^{m-k} \left[ k\partial_{\lambda}^{k-1} + \lambda \partial_{\lambda }^k \right] f(\lambda ) \quad \\\
= & \left[ m\partial_{\lambda}^{m-1} + \lambda \partial_{\lambda}^m \right] f(\lambda),
\end{align*}
we can write that
\begin{align*}
Pr\{ T^{\alpha, \beta}_k \leq t \} = & \sum_{m=k}^\infty \frac{(-\lambda)^m}{m!} \left[ m \partial_\lambda^{m-1} + \lambda \partial_\lambda^m \right]  \frac{\mathbb{E} e^{-\lambda \mathfrak{F}^{\alpha, \beta}_t}}{\lambda}\\
= & \sum_{m=k} \left[ -\lambda \frac{(-\lambda \partial_\lambda)^{m-1}}{(m-1)!} + \lambda \frac{(-\lambda \partial_\lambda)^m}{m!} \right] \frac{\mathbb{E} e^{-\lambda \mathfrak{F}^{\alpha, \beta}_t}}{\lambda}\\
= & \left[ -\lambda \sum_{m=k-1}^\infty \frac{(-\lambda \partial_\lambda)^m}{m!} + \lambda \sum_{m=k}^\infty \frac{(-\lambda \partial_\lambda)^m}{m!} \right] \frac{\mathbb{E} e^{-\lambda \mathfrak{F}^{\alpha, \beta}_t}}{\lambda}\\
= & - \lambda \frac{(-\lambda \partial_\lambda)^{k-1}}{(k-1)!} \frac{\mathbb{E} e^{-\lambda \mathfrak{F}^{\alpha, \beta}_t}}{\lambda}.
\end{align*}
Therefore, we get
\begin{align*}
Pr\{ T^{\alpha, \beta}_k \in dt \} / dt = & - \lambda \frac{(-\lambda \partial_\lambda)^{k-1}}{(k-1)!} \frac{\partial}{\partial t} \frac{\mathbb{E} e^{-\lambda \mathfrak{F}^{\alpha, \beta}_t}}{\lambda}\\
= & [\textrm{by } \eqref{dertF}]\\
= & -\lambda \frac{(-\lambda \partial_\lambda)^{k-1}}{(k-1)!} \frac{\beta}{\alpha t} \partial_{\lambda} \mathbb{E} e^{-\lambda \mathfrak{F}^{\alpha, \beta}_t}\\
= & \frac{\beta k}{\alpha t}\frac{(-\lambda \partial_\lambda)^{k}}{k!}\mathbb{E} e^{-\lambda \mathfrak{F}^{\alpha, \beta}_t}\\
= & \frac{\beta k}{\alpha t} p_k(t;\alpha, \beta).
\end{align*}

As a second step we show that
\begin{equation*}
T^{\alpha, \beta}_k  \stackrel{law}{=} \mathfrak{F}^{\beta, \alpha}_{T_k}, \quad k =1,2, \ldots .
\end{equation*}
Let us consider the density
\begin{align*}
Pr\{ \mathfrak{F}^{\beta, \alpha}_{T_k} \in dt \}/dt = & \frac{\lambda^k}{(k-1)!} \int_0^\infty z^{k-1} e^{-\lambda z} f_z^{\beta, \alpha}(t) dz \\
= &  \frac{\lambda}{(k-1)!} \int_0^\infty (-\lambda \partial_{\lambda})^{k-1} e^{-\lambda z} f_z^{\beta, \alpha}(t)dz
\end{align*}
where
\begin{align*}
\int_0^\infty  e^{-\lambda z} f_z^{\beta, \alpha}(t)dz = & \int_0^\infty e^{-\lambda z} \int_0^\infty h_\beta(t,s)l_\alpha(s, z)dsdz\\
= & \int_0^\infty h_\beta(t,s) \lambda^{\alpha -1} e^{-s\lambda^\alpha}ds\\
= & \lambda^{\alpha-1} t^{\beta-1} E_{\beta, \beta}(-t^\beta \lambda^\alpha).
\end{align*}
In the last steps we used formulae \eqref{lap-inv} and \eqref{lap-sub}. Therefore, we get that
\begin{align*}
Pr\{ \mathfrak{F}^{\beta, \alpha}_{T_k} \in dt \}/dt = &\frac{\lambda}{(k-1)!} (-\lambda \partial_{\lambda})^{k-1} \lambda^{\alpha-1} t^{\beta-1} E_{\beta, \beta}(-t^\beta \lambda^\alpha)\\
= & \frac{\beta}{\alpha t}\frac{(-\lambda \partial_\lambda)^k}{ (k-1)!} E_\beta(-t^\beta \lambda^\alpha)\\
=&\frac{\beta k}{\alpha t} p_k(t; \alpha, \beta).
\end{align*}
This concludes the proof.
\end{proof}

\begin{os}
Orsingher and Polito \cite{ors-pol-SPL} proved that the solution to %
\eqref{eq-Pois-frac-time-space} can be written as follows
\begin{equation}
p_k(t; \alpha, \beta)=\frac{(-1)^{k}}{k!}\sum_{r=0}^{\infty }\frac{(-\lambda ^{\alpha }t^{\beta
})^{r}}{\Gamma (\beta r+1)}\frac{\Gamma (r+1)}{\Gamma (\alpha r+1-k)},\quad
k\geq 0,\;\alpha \in (0,1],\;\beta \in (0,1].  \label{ors-pol-sol}
\end{equation}%
After some calculation, we can see that \eqref{ors-pol-sol} coincides with
our compact representation given in \eqref{probNsub}.
\end{os}

 As special cases for $\alpha =1$ or $\beta =1$, we can obtain
from Theorem 2 some results on well-known processes. Indeed, in the first case, the subordinated
Poisson process coincides with the time-fractional Poisson process studied
in \cite{Rep, BegOrs, Mai, Mee}).

\begin{coro}
For $\alpha =1$, the probability law of the time-changed Poisson process $N(%
\mathfrak{F}_{t}^{1,\beta })=N(\mathfrak{L}_{t}^{\beta }),$ $t>0,\;\beta \in
(0,1]$, i.e.%
\begin{equation*}
p_{k}(t;1,\beta )=P\{N(\mathfrak{L}_{t}^{\beta })=k\}=\frac{(-\lambda
\partial _{\lambda })^{k}}{k!}E_{\beta }(-t^{\beta }\lambda )
\end{equation*}%
satisfies the following equation
\begin{equation*}
\left( \mathcal{D}^\beta_t +\lambda (I-B)\right) p_{k}(t;1,\beta
)=0,\quad k\in \mathbb{N}_{0},\;t>0
\end{equation*}%
subject to the initial condition
\begin{equation*}
p_{k}(0;1,\beta )=\left\{
\begin{array}{ll}
0, & k\geq 1, \\
1, & k=0.%
\end{array}%
\right.
\end{equation*}%
Furthermore, for the hitting time we have that
\begin{equation}
T_{k}^{1,\beta }\stackrel{law}{=}\mathfrak{A}_{T_{k}}^{\beta }  \label{waiting-T-beta}
\end{equation}%
and
\begin{equation*}
Pr\{T_{k}^{1,\beta }\in dt\}/dt=\frac{\beta }{t}\frac{(-\lambda
\partial _{\lambda })^{k}}{(k-1)!}E_{\beta }(-t^{\beta }\lambda ).
\end{equation*}
\end{coro}

Moreover, for the subordinated Poisson process $N(\mathfrak{F}_{t}^{1,\beta
})$, we can prove the following result on its hitting time %
\eqref{waiting-T-beta}.

\begin{te}
The following holds
\begin{equation*}
\sum_{k=0}^{N_{t}}T_{k}^{1,\beta }=\sum_{k=0}^{N_{t}}\mathfrak{A}%
_{T_{k}}^{\beta }=\int_{0}^{t}\mathfrak{A}_{s}^{\beta }\,dN_{s}\overset{law}{%
=}\mathfrak{A}_{\mathfrak{T}_{t}}^{\beta }
\end{equation*}%
where
\begin{equation*}
\mathfrak{T}_{t}=\int_{0}^{t}s\,dN_{s}
\end{equation*}
\end{te}

\begin{proof}
From the fact that
\begin{equation*}
 \sum_{N_t > k} \mathfrak{A}^\beta_{T_k} =  \sum_{T_k < t} \mathfrak{A}^\beta_{T_k}
\end{equation*}
we can write
\begin{align*}
\mathbb{E} \exp\left( - \xi \sum_{T_k < t} \mathfrak{A}^\beta_{T_k} \right) = & \mathbb{E}\Bigg[ \prod_{T_k < t} \mathbb{E} \left[\exp\left( - \xi \mathfrak{A}^\beta_{T_k} \right) \Big| T_k \right] \Bigg]\\
= & \mathbb{E}\left[ \prod_{T_k < t} \exp\left( - \xi^\beta T_k \right)\right] \\
= & \mathbb{E} \exp\left( - \xi^\beta \sum_{T_k < t} T_k \right)\\
= & \mathbb{E} \exp\left(-\xi \mathfrak{A}^\beta_{\mathfrak{T}_t} \right)
\end{align*}
where
\begin{equation*}
\mathfrak{T}_t = \sum_{T_k < t} T_k = \int_0^t s\, dN_s
\end{equation*}
and this concludes the proof.
\end{proof}

 In the other special case, i.e. for $\beta =1$, the time-changed Poisson process $N(%
\mathfrak{F}_{t}^{\alpha ,\beta})$ reduces to the space-fractional Poisson process studied in \cite{ors-pol-SPL}.

\begin{coro}
For $\beta =1,$ the probability law of the time-changed Poisson process $N(%
\mathfrak{F}_{t}^{\alpha ,1})=N(\mathfrak{A}_{t}^{\alpha }),$ $t>0,\;\alpha
\in (0,1]$, i.e.
\begin{equation}
p_{k}(t;\alpha ,1)=P\{N(\mathfrak{A}_{t}^{\alpha })=k\}=\frac{(-\lambda
\partial _{\lambda })^{k}}{k!}e^{-t\lambda ^{\alpha }}  \label{sf2}
\end{equation}%
satisfies the following equation
\begin{equation*}
\left( \frac{d}{dt}+\lambda^\alpha (I-B)^{\alpha }\right) p_{k}(t;\alpha
,1)=0,\quad k=0,1,2\ldots ,\;t>0
\end{equation*}%
subject to the initial condition
\begin{equation*}
p_{k}(0;\alpha ,1)=\left\{
\begin{array}{ll}
0, & k\geq 1, \\
1, & k=0.%
\end{array}%
\right.
\end{equation*}%
Furthermore, the hitting time can be written as
\begin{equation*}
T_{k}^{\alpha ,1}=\mathfrak{L}_{T_{k}}^{\alpha }
\end{equation*}%
and
\begin{equation*}
Pr\{T_{k}^{\alpha ,1}\in dt\}/dt=\frac{1}{\alpha t}\frac{(-\lambda
\partial _{\lambda })^{k}}{(k-1)!}e^{-t\lambda ^{\alpha }}.
\end{equation*}
\end{coro}

We prove now that the process $N(\mathfrak{A}_{t}^{\alpha })$ is governed
also by an alternative fractional differential equation, of order $1/\alpha
>1$. In this case we use the left-sided Riemann-Liouville fractional
derivative, defined as%
\begin{equation*}
\frac{d^{\nu }}{d(-x)^{\nu }}f(x)=\frac{1}{\Gamma (m-\nu )}\left( -\frac{d}{%
dx}\right) ^{m}\int_{x}^{+\infty }\frac{f(s)ds}{(s-x)^{1+\nu -m}}\quad
m-1<\nu <m.
\end{equation*}

\begin{te}
The distribution $p_{k}(t;\alpha ,1)$ of $N(\mathfrak{A}_{t}^{\alpha })$
is the solution to the following equation%
\begin{equation}
\left( \frac{d^{\frac{1}{\alpha }}}{d(-t)^{\frac{1}{\alpha }}}-\lambda
(I-B)\right) p_{k}(t;\alpha ,1)=0,\quad k=0,1,2\ldots ,\;t>0,  \quad \alpha \in (0,1] \label{ma}
\end{equation}%
subject to the initial conditions
\begin{equation*}
p_{k}(0;\alpha ,1)=\left\{
\begin{array}{ll}
0, & k\geq 1, \\
1, & k=0.%
\end{array}%
\right.
\end{equation*}%
and
\begin{equation}
\left. \frac{d^{j}}{dt^{j}}p_{k}(t;\alpha ,1)\right\vert _{t=0}=(-1)^{k}%
\frac{\lambda ^{\alpha j}}{k!}\frac{\Gamma (\alpha j+1)}{\Gamma (\alpha
j+1-k)},\quad j=1,...\left\lfloor 1/\alpha \right\rfloor -1.  \label{con2}
\end{equation}
\end{te}

\begin{proof} We start by proving that the density $h_{\alpha }(x,t)$ of the
subordinator $\mathfrak{A}_{t}^{\alpha }$ satisfies the fractional
differential equation (of order $1/\alpha $ greater than one)%
\begin{equation}
\left( \frac{\partial ^{1/\alpha }}{\partial (-t)^{1/\alpha }}-\frac{%
\partial }{\partial x}\right) h_{\alpha }(x,t)=0,\quad x,t\geq 0,
\label{dov}
\end{equation}%
with initial conditions%
\begin{equation}
\left\{
\begin{array}{l}
h_{\alpha }(x,0)=\delta (x),\quad x\geq 0 \\
h_{\alpha }(0,t)=0,\quad t\geq 0 \\
\left. \frac{\partial ^{j}}{\partial t^{j}}h_{\alpha }(x,t)\right\vert
_{t=0}=(-1)^{j}\Phi _{\alpha j+1}(x),\quad j=1,...,\left\lfloor 1/\alpha
\right\rfloor -1%
\end{array}%
\right.   \label{dov-2}
\end{equation}%
where $\Phi _{\alpha j+1}(z)=\frac{z^{-\alpha j-1}}{\Gamma (-\alpha j)},$
for $z>0.$ For the Laplace transform of (\ref{dov}), we get%
\begin{eqnarray*}
&&\int_{0}^{\infty }e^{-\xi x}\frac{\partial ^{1/\alpha }}{\partial
(-t)^{1/\alpha }}h_{\alpha }(x,t)dx = [\text{by \eqref{lap-sub}}]\\
&=&\frac{\partial ^{1/\alpha }}{\partial (-t)^{1/\alpha }}e^{-\xi ^{\alpha
}t}=[\text{by (2.2.15) in \cite{Kilb}}] \\
&=&\xi e^{-\xi ^{\alpha }t}=[\text{by (\ref{dov-2})}]=\int_{0}^{\infty
}e^{-\xi x}\frac{\partial }{\partial x}h_{\alpha }(x,t)dx.
\end{eqnarray*}%
The third condition in (\ref{dov-2}) can be checked by noting that%
\begin{eqnarray*}
&&\int_{0}^{\infty }e^{-\xi x}\left. \frac{\partial ^{j}}{\partial t^{j}}%
h_{\alpha }(x,t)\right\vert _{t=0}dx \\
&=&\left. \frac{\partial ^{j}}{\partial t^{j}}e^{-\xi ^{\alpha
}t}\right\vert _{t=0}=(-1)^{j}\xi ^{\alpha j} \\
&=&(-1)^{j}\int_{0}^{\infty }e^{-\xi x}\frac{x^{-\alpha j-1}}{\Gamma
(-\alpha j)}dx=(-1)^{j}\int_{0}^{\infty }e^{-\xi x}\Phi _{\alpha j+1}(x)dx,
\end{eqnarray*}%
while the others are immediately verified. The result given in (\ref{dov})-(%
\ref{dov-2}) generalizes, to any $\alpha \in (0,1)$, Theorem 2 in \cite{Dov},
which has been proved in the special case $\alpha =1/n.$

We now prove equation (\ref{ma}):
\begin{eqnarray*}
&&\frac{d^{\frac{1}{\alpha }}}{d(-t)^{\frac{1}{\alpha }}}p_{k}(t;\alpha ,1)
\\
&=&\frac{\partial ^{\frac{1}{\alpha }}}{\partial (-t)^{\frac{1}{\alpha }}}%
\int_{0}^{\infty }p_{k}(z)h_{\alpha }(z,t)dz=[\text{by (\ref{dov})]} \\
&=&\int_{0}^{\infty }p_{k}(z)\frac{\partial }{\partial z}h_{\alpha }(z,t)dz=[%
\text{by (\ref{dov-2}]} \\
&=&-\int_{0}^{\infty }\frac{d}{dz}p_{k}(z)h_{\alpha }(z,t)dz=\lambda
(I-B)p_{k}(t;\alpha ,1).
\end{eqnarray*}%
Condition (\ref{con2}) is obtained as follows%
\begin{eqnarray*}
&&\left. \frac{d^{j}}{dt^{j}}p_{k}(t;\alpha ,1)\right\vert _{t=0} \\
&=&\int_{0}^{\infty }p_{k}(z)\left. \frac{d^{j}}{dt^{j}}h_{1/\alpha
}(z,t)\right\vert _{t=0}dz=[\text{by (\ref{dov-2})]} \\
&=&\frac{\lambda ^{k}}{k!}\int_{0}^{\infty }e^{-\lambda z}\frac{z^{k-\alpha
j-1}}{\Gamma (-\alpha j)}dz=\frac{\lambda ^{\alpha j}}{k!}\frac{\Gamma
(-\alpha j+k)}{\Gamma (-\alpha j)} \\
&=&\frac{\lambda ^{\alpha j}}{k!}\frac{\Gamma (1+\alpha j)}{\Gamma (1+\alpha
j-k)}\frac{\sin (-\pi \alpha j)}{\sin (-\pi \alpha j+\pi k)} \\
&=&(-1)^{k}\frac{\lambda ^{\alpha j}}{k!}\frac{\Gamma (1+\alpha j)}{\Gamma
(1+\alpha j-k)},\quad j=1,...,\left\lfloor 1/\alpha \right\rfloor -1
\end{eqnarray*}%
and it is satisfied by $p_{k}(t;\alpha ,1)$, as can be checked by
differentiating formula (\ref{ors-pol-sol}). The other conditions
are immediately verified.
\end{proof}

Let us introduce the following differential operator
\begin{equation}
\label{dif-op-X}
\mathbb{D}_\psi\, f(x) = \int_0^\infty (P_s f(x) - f(x))M(ds)
\end{equation}
where $P_s$ is the semigroup (of a L\'{e}vy process previously introduced) associated to the infinitesimal generator $\mathcal{A}$ and $M(\cdot)$ is the L\'{e}vy measure of the subordinator $\mathfrak{D}_t$, $t>0$ with symbol $\psi$. From \eqref{symb-sub-eta}, we immediately get that
\begin{equation}
\label{symb-D-psi}
\widehat{\mathbb{D}_\psi\, f}(\xi) = \int_0^\infty (e^{-s \Psi(\xi)} \widehat{f}(\xi) - \widehat{f}(\xi))M(ds) = \psi(\Psi(\xi)) \widehat{f}(\xi).
\end{equation}
Indeed, $e^{-t\Psi}$ is the symbol of $P_t=e^{t\mathcal{A}}$. An alternative form can be given as follows
\begin{align*}
\mathbb{D}_\psi\, f(x) = & \int_0^\infty (P_s f(x) - f(x))M(ds)\\
= & \int_0^\infty (\mathbb{E} f(X^x_s) - f(x))M(ds)\\
= & \int_0^\infty \int_{\mathbb{R}}(f(y) - f(x)) Pr\{ X^x_s \in dy \} M(ds)\\
= &  \int_{\mathbb{R}}(f(y) - f(x)) J(x,y)dy
\end{align*}
when the integral $J(x,y)dy = \int_0^\infty Pr\{ X^x_s \in dy \} M(ds)$ converges. We now consider the differential operator \eqref{dif-op-X} where $P_s$ is the semigroup of the Poisson process $N(t)$, $t>0$.

Our aim in the present paper is to investigate the time-changed Poisson processes presented so far. Nevertheless, we are able to state the following general result which refers to the case where a general subordinator $\mathfrak{D}_t$, $t>0$, is used as time argument.

\begin{te}
Let $\mathfrak{D}_t$, $t\geq 0$ be a subordinator with symbol \eqref{symb-sub-eta}. Let $\mathfrak{L}^\beta_t$, $t \geq 0$ be the inverse to a stable subordinator of order $\beta \in (0,1]$. The time-changed Poisson process
\begin{equation}
N(\mathfrak{D}_{\mathfrak{L}^\beta_t}), \quad t\geq 0
\end{equation}
has probability distribution function
\begin{equation}
\label{law-NDL}
p^\psi_k(t;\beta) = \frac{(-\lambda \partial_\lambda)^k}{k!} E_\beta(-t^\beta \psi(\lambda)), \quad k=0,1,2, \ldots .
\end{equation}
Furthermore, \eqref{law-NDL} is the solution to
\begin{equation}
\left( \mathcal{D}^\beta_t - \mathbb{D}_\psi \right) p^\psi_k(t;\beta) =0, \quad t>0, \; k=0,1,2, \ldots
\end{equation}
subject to the initial condition
\begin{equation*}
p_k^\psi(0; \beta) = \left\lbrace \begin{array}{ll}
1, & k=0,\\
0, & k\geq 1
\end{array} \right .
\end{equation*}
where
\begin{equation}
\mathbb{D}_\psi u_k= - \int_0^\infty  \mathbb{E} \left( u_k - B^{N(s)} u_{k} \right) M(ds) = -\psi\left(\lambda(I-B)\right)u_k. \label{op-DNL-integral}
\end{equation}
and $Bu_k=u_{k-1}$ as usual.
\end{te}
\begin{proof}
We have that
\begin{equation}
\mathbb{E} \exp \left( -\lambda \mathfrak{D}_{\mathfrak{L}^\beta_t} \right) = \mathbb{E} \exp\left( - \psi(\lambda) \mathfrak{L}^\beta_t\right) = E_\beta(-t^\beta \psi(\lambda)) \label{lap-NDL}
\end{equation}
and therefore
\begin{align*}
p^\psi_k(t;\beta) = & \frac{(-\lambda \partial_\lambda)^k}{k!} \mathbb{E} \exp \left( -\lambda \mathfrak{D}_{\mathfrak{L}^\beta_t} \right) =  \mathbb{E} \left[ \frac{(-\lambda \partial_\lambda)^k}{k!} \exp \left( -\lambda \mathfrak{D}_{\mathfrak{L}^\beta_t} \right) \right]\\
= & \int_0^\infty \frac{(-\lambda \partial_\lambda)^k}{k!} e^{-\lambda s} P\left\lbrace \mathfrak{D}_{\mathfrak{L}^\beta_t} \in ds\right\rbrace = \int_0^\infty P\{N(s) =k \} P\left\lbrace \mathfrak{D}_{\mathfrak{L}^\beta_t} \in ds\right\rbrace\\
= & P\{N(\mathfrak{D}_{\mathfrak{L}^\beta_t})=k\}.
\end{align*}
Furthermore, from \eqref{lap-NDL},
\begin{equation*}
\mathbb{E}\exp\left(i \xi N(\mathfrak{D}_{\mathfrak{L}^\beta_t}) \right)= \mathbb{E} \exp\left(- \lambda (1- e^{i\xi}) \mathfrak{D}_{\mathfrak{L}^\beta_t}) \right) = E_\beta\Big(-t^\beta \psi\left(\lambda(1-e^{i\xi})\right)\Big).
\end{equation*}
It is well-known that $E_\beta$ is an eigenfunction for $\mathcal{D}^\beta_t$; in particular, from \eqref{symb-D-psi} we can write
\begin{equation*}
\mathcal{D}^\beta_t\, E_\beta\Big(-t^\beta \psi\left(\lambda(1-e^{i\xi})\right)\Big) = \psi\left(\lambda(1-e^{i\xi})\right) \,  E_\beta\Big(-t^\beta \psi\left(\lambda(1-e^{i\xi})\right)\Big).
\end{equation*}
The symbol $\psi\left(\lambda(1-e^{i\xi})\right)$ is the Fourier multiplier of  $\mathbb{D}_\psi$ when $\Psi(\xi) = \lambda(1-e^{i\xi})$, that is the L\'{e}vy process $X$ is the Poisson process $N$.

We now show that \eqref{op-DNL-integral} holds true. It suffices to see that
\begin{align*}
\mathbb{E} \left( u_k - B^{N(s)}u_k \right) = &  \left( u_k - \mathbb{E} B^{N(s)}u_k \right)\\
= & \left( u_k - e^{-s\lambda} \sum_{n=0}^{\infty} \frac{(s\lambda B)^{n}}{n!} u_k \right)\\
= & \left( u_k - e^{-s\lambda (I - B)} u_k \right)\\
= & \left( u_k - P_s u_k \right).
\end{align*}
From \eqref{symb-D-psi} (or \eqref{symb-sub-eta}) we get the claim.
\end{proof}

\begin{os}
Concerning the solution \eqref{law-NDL}, it is worth to recall that
\begin{equation*}
(-\partial_z)^k \psi(z) \leq 0 \quad \textrm{and} \quad (-\partial_z)^k E_\beta(-z) \geq 0.
\end{equation*}
Indeed, $\psi$ is a Bernstein function and therefore its derivative is completely monotone, whereas the Mittag-Leffler
is completely monotone. Furthermore, $E_\beta(0)=1$ and there exists a unique probability measure $m$ such that $E_\beta(-\xi) = \int_0^\infty e^{-\xi x}m(dx)$. In particular, we have seen before that $m(dx)=l_\beta(x,1)dx$. We notice that, if $f$ is a Bernstein function on $(0,\infty)$, then $g(f)$ is completely monotone for every completely monotone $g$ (see Theorem 3.6 of \cite{book-song-et-al}). This means that $p^\psi_k \geq 0$ for all $k\geq 0$.
\end{os}

\section{Poisson process with random drift}
\label{SecPoisDrift}

We consider now the process defined in \eqref{dr3} where $\mathfrak{F}^{\gamma, \beta}_t = \mathfrak{A}^\gamma_{\mathfrak{L}^\beta_t}$ and $\mathfrak{F}^{\alpha, \beta}_t = \mathfrak{A}^\alpha_{\mathfrak{L}^\beta_t}$ are independent from $N(t)$.

\begin{te}
The drifted process
\begin{equation}
N(\mathfrak{F}^{\gamma,\beta}_t) + a \mathfrak{F}^{\alpha, \beta}_t, \quad
t>0, \; a \geq 0, \; \alpha, \gamma, \beta \in (0, 1]
\label{sub-pois-with-sub-drift}
\end{equation}
has probability law
\begin{align}
\Pr\{N(\mathfrak{F}^{\gamma,\beta}_t) + a \mathfrak{F}^{\alpha,
\beta}_t \in dx\}/dx = & \sum_{k=0}^\infty \frac{(-\lambda
\partial_\lambda)^k}{k!} \int_0^\infty e^{-s\lambda^\gamma} h_\alpha(x-k,
a^\alpha s) l_\beta(s,t)ds \notag \\
= & \sum_{k=0}^\infty \frac{(-\lambda \partial_\lambda)^k}{k!}
\mathbb{E} \left[ \exp (-\lambda^\gamma \mathfrak{L}^\beta_t)\,
h_\alpha(x-k, a^\alpha \mathfrak{L}^\beta_t)\right]  \label{semig-rep}
\end{align}
which is the solution to the equation
\begin{equation}
\left(\mathcal{D}^\beta_t + a^\alpha \partial^\alpha_x
+ \lambda^\gamma \left(I - K \right)^\gamma \right) u(x,t)=0, \quad x \in
\mathbb{R}_0^+,\; t>0  \label{genEq}
\end{equation}
with initial condition $u(x,0)=\delta(x)$, where
\begin{equation}
(I-K)^\gamma = \sum_{j=0}^\infty (-1)^j \binom{\gamma}{j} K^j \label{I-K-power}
\end{equation}
and
\begin{equation*}
K^j = \left\lbrace
\begin{array}{ll}
e^{-j\partial_x}, & \text{if }\; a >0 \\
B^j, & \text{if }\; a=0%
\end{array}
\right .
\end{equation*}
is the shift operator.
\label{theorem-semig-N}
\end{te}

\begin{os}
We observe that the solution \eqref{semig-rep} can be represented as follows
\begin{equation*}
\Pr\{N(\mathfrak{F}^{\gamma,\beta}_t) + a \mathfrak{F}^{\alpha,
\beta}_t \in dx\}/dx = e^{\lambda t} P_t \varphi(x, t)
\end{equation*}
where $P_t$ is the Poisson semigroup \eqref{Pois-semigroup} and
\begin{align*}
\varphi(x,t) = & \mathbb{E} \left[ \exp (-\lambda^\gamma \mathfrak{L}^\beta_t)\,
h_\alpha(x, a^\alpha \mathfrak{L}^\beta_t)\right].
\end{align*}
By using the fact that $h_\alpha(x,t) = \frac{\alpha t}{x} l_\alpha(t,x)$ where $l_\alpha$ is the density law \eqref{density-l}, we get
\begin{align*}
\varphi(x,t) = & \mathbb{E} \left[ \frac{a^\alpha \mathfrak{L}^\beta_t}{x} \exp (-\lambda^\gamma \mathfrak{L}^\beta_t)\, l_\alpha(a^\alpha \mathfrak{L}^\beta_t, x) \right].
\end{align*}

\end{os}

\begin{proof}[Proof of Theorem \ref{theorem-semig-N}]
First we observe that
\begin{align*}
\mathbb{E}e^{-\xi N(\mathfrak{F}^{\gamma,\beta}_t) - a \xi \mathfrak{F}^{\alpha, \beta}_t} = & \mathbb{E}\Bigg[ \mathbb{E} \left[e^{-\xi N(\mathfrak{F}^{\gamma,\beta}_t)} e^{ - a \xi \mathfrak{F}^{\alpha, \beta}_t} \Big| \mathfrak{L}^\beta_t \right] \Bigg]
\end{align*}
where
\begin{equation*}
 \mathbb{E}\left[ e^{ - a \xi \mathfrak{F}^{\alpha, \beta}_t}\Big| \mathfrak{L}^\beta_t \right] = \exp\left( - a^\alpha \xi^\alpha \mathfrak{L}^\beta_t \right)
\end{equation*}
and
\begin{align*}
\mathbb{E}\left[e^{-\xi N(\mathfrak{F}^{\gamma,\beta}_t)} \Big| \mathfrak{L}^\beta_t \right]= & \mathbb{E} \left[ \mathbb{E} \left(e^{-\xi N(\mathfrak{F}^{\gamma,\beta}_t)} \big| \mathfrak{F}^{\gamma,\beta}_t \right) \Big| \mathfrak{L}^\beta_t \right]\\
= & \mathbb{E}\left[ \exp\left( - \lambda (1- e^{-\xi}) \mathfrak{F}^{\gamma,\beta}_t \right)\Big| \mathfrak{L}^\beta_t \right]\\
= & \exp\left( -\lambda^\gamma (1-e^{-\xi})^\gamma  \mathfrak{L}^\beta_t \right)
\end{align*}
Therefore we obtain that
\begin{align}
\mathbb{E}e^{-\xi N(\mathfrak{F}^{\gamma,\beta}_t) - a \xi \mathfrak{F}^{\alpha, \beta}_t} = & \mathbb{E} \exp\left( - \xi^\alpha \mathfrak{L}^\beta_t -\lambda^\gamma (1-e^{-\xi})^\gamma  \mathfrak{L}^\beta_t \right)\notag \\
= & E_\beta\left( -t^\beta a^\alpha \xi^\alpha - t^\beta \lambda^\gamma (1-e^{-\xi})^\gamma \right). \label{lap-gen-theorem}
\end{align}
The density can be obtained by considering that
\begin{align*}
\Pr\{N(\mathfrak{F}^{\gamma,\beta}_t) + a \mathfrak{F}^{\alpha, \beta}_t \in A\}= &  \sum_{k=0}^\infty \Pr\left\lbrace\mathfrak{F}^{\alpha, \beta}_t \in \frac{A-k}{a}\right\rbrace P\{ N(\mathfrak{F}^{\gamma,\beta}_t)  =k \}
\end{align*}
for every Borel set $A \in \mathcal{B}(\mathbb{R_+})$ and
\begin{align*}
\Pr\{N(\mathfrak{F}^{\gamma,\beta}_t) + a \mathfrak{F}^{\alpha, \beta}_t \in A \big| \mathfrak{L}^\beta_t = s \}= &  \sum_{k=0}^\infty \Pr\left\lbrace\mathfrak{A}^{\alpha}_s \in \frac{A-k}{a}\right\rbrace P\{ N(\mathfrak{A}^{\gamma}_s)  =k \}\\
= & \sum_{k=0}^\infty \frac{(-\lambda \partial_\lambda)^k}{k!} e^{-s \lambda^\gamma} \Pr\left\lbrace\mathfrak{A}^{\alpha}_s \in \frac{A-k}{a}\right\rbrace \\
= & \sum_{k=0}^\infty \frac{(-\lambda \partial_\lambda)^k}{k!} e^{-s \lambda^\gamma} \Pr\left\lbrace\mathfrak{A}^{\alpha}_{a^\alpha s} \in A-k\right\rbrace \\
= & \int_A \sum_{k=0}^\infty \frac{(-\lambda \partial_\lambda)^k}{k!} e^{-s \lambda^\gamma} h_\alpha(x-k, a^\alpha s)dx.
\end{align*}
By integrating with respect to $Pr\{ \mathfrak{L}^\beta_t \in ds \}$, we obtain
\begin{align*}
\Pr\{N(\mathfrak{F}^{\gamma,\beta}_t) + a \mathfrak{F}^{\alpha, \beta}_t \in A\}= &  \int_A \sum_{k=0}^\infty \frac{(-\lambda \partial_\lambda)^k}{k!} \int_0^\infty e^{-s \lambda^\gamma} h_\alpha(x-k, a^\alpha s) l_\beta(s, t)ds\,dx.
\end{align*}
and thus
\begin{equation}
\Pr\{N(\mathfrak{F}^{\gamma,\beta}_t) + a \mathfrak{F}^{\alpha, \beta}_t \in dx\}= \sum_{k=0}^\infty \frac{(-\lambda \partial_\lambda)^k}{k!} \int_0^\infty e^{-s \lambda^\gamma} h_\alpha(x-k, a^\alpha s) l_\beta(s, t)ds\,dx.
\end{equation}
A further check involves the Laplace transform
\begin{align*}
& \int_0^\infty \int_0^\infty e^{-\xi x}  e^{-s \lambda^\gamma} h_\alpha(x-k, a^\alpha s) l_\beta(s, t)ds\,dx\\
= & e^{-\xi k} \int_0^\infty  e^{-s\lambda^\gamma - a^\alpha s \xi^\alpha}l_\beta(s, t)ds = e^{-\xi k} E_\beta(-t^\beta (\lambda^\gamma + a^\alpha \xi^\alpha)).
\end{align*}
Due to the action of the shift operator we arrive at
\begin{align*}
\int_0^\infty e^{-\xi x} \Pr\{N(\mathfrak{F}^{\gamma,\beta}_t) + a \mathfrak{F}^{\alpha, \beta}_t \in dx\} = & \sum_{k=0}^\infty \frac{(-\lambda \partial_\lambda)^k}{k!} e^{-\xi k} E_\beta(-t^\beta (\lambda^\gamma + a^\alpha \xi^\alpha))\\
= & e^{- e^{-\xi} \lambda \partial_\lambda}E_\beta(-t^\beta (\lambda^\gamma + a^\alpha \xi^\alpha))\\
= & E_\beta(-t^\beta (\lambda^\gamma(1- e^{-\xi})^\gamma + a^\alpha \xi^\alpha))
\end{align*}
which coincides with \eqref{lap-gen-theorem}.

We now study the governing equation \eqref{genEq}. From representation \eqref{I-K-power}, where $K$ is the translation operator for both continuous and discrete supported functions, we get that
\begin{align*}
\int_0^\infty e^{-\xi x}(I-K)^\gamma u(x,t)dx = & \sum_{j=0}^\infty (-1)^j \binom{\gamma}{j}\int_0^\infty e^{-\xi x} K^j u(x,t)dx\\
= & \sum_{j=0}^\infty (-1)^j \binom{\gamma}{j}\int_0^\infty e^{-\xi x} u(x-j,t)dx\\
= & \sum_{j=0}^\infty (-1)^j \binom{\gamma}{j} e^{-j\xi}\int_0^\infty e^{-\xi x} u(x, t)dx\\
= & (1 - e^{-\xi})^\gamma \int_0^\infty e^{-\xi x} u(x, t)dx.
\end{align*}
Let us write
\begin{equation*}
\widetilde{\widetilde{u}}(\xi, \mu) = \int_0^\infty e^{-\mu t }\int_0^\infty e^{-\xi x} u(x, t)dx dt.
\end{equation*}
Equation \eqref{genEq} becomes
\begin{align*}
\mu^\beta \widetilde{\widetilde{u}}(\xi, \mu) - \mu^{\beta -1}\widetilde{u_0}(\xi) + a^\alpha \xi^\alpha  \widetilde{\widetilde{u}}(\xi, \mu) = -\lambda^\gamma (I- e^{-\xi})^\gamma \widetilde{\widetilde{u}}(\xi, \mu)
\end{align*}
from which we obtain
\begin{equation}
\widetilde{\widetilde{u}}(\xi, \mu) = \frac{\mu^{\beta -1}}{\mu^\beta + a^\alpha \xi^\alpha + \lambda^\gamma(1-e^{-\xi})^\gamma} \widetilde{u_0}(\xi) \label{doubleLNdrift2}
\end{equation}
where, obviously
\begin{equation*}
\widetilde{u_0}(\xi) = \int_0^\infty e^{-\xi x} u(x, 0)dx.
\end{equation*}
From \eqref{lap-gen-theorem} we obtain the double Laplace transform
\begin{equation}
\int_0^\infty e^{-\mu t} \mathbb{E}\left[e^{-\xi N(\mathfrak{F}^{\gamma,\beta}_t) - a \xi \mathfrak{F}^{\alpha, \beta}_t} \right]dt = \frac{\mu^{\beta -1}}{\mu^\beta + a^\alpha \xi^\alpha + \lambda^\gamma(1-e^{-\xi})^\gamma}. \label{doubleLNdrift1}
\end{equation}
For $u_0=\delta$, formula \eqref{doubleLNdrift2} coincides with \eqref{doubleLNdrift1} and therefore we obtain the claimed result.
\end{proof}

For $\beta =1$, the composition \eqref{sub-pois-with-sub-drift} becomes
\begin{equation}
N(\mathfrak{A}_{t}^{\gamma })+a\mathfrak{A}_{t}^{\alpha },\quad t>0 \label{N-time-changed-A}
\end{equation}
and coincides, for $\gamma =1$, with $N(t)+a\mathfrak{A}_{t}^{\alpha }.$ For
the latter we present the following results, concerning its governing
equation and its hitting time. We remark that, for $\alpha =1,$ it reduces
to the drifted Poisson process \eqref{PoisDrift} whereas, for $a=0,$ it
coincides with the standard Poisson process.

\begin{coro}
The Poisson process with subordinated drift
\begin{equation}
N(t)+a\mathfrak{A}_{t}^{\alpha },\quad t>0,\;a\geq 0,\;\alpha \in (0,1]
\label{pois-with-sub-drift-coro}
\end{equation}%
has probability law
\begin{equation}
Pr\{N(t)+a\mathfrak{A}_{t}^{\alpha }\in dx\}/dx=e^{-\lambda
t}\sum_{k=0}^{\infty }\frac{(\lambda t)^{k}}{k!}h_{\alpha }\left(
x-k,a^{\alpha }t\right) \,\mathbf{1}_{(k<x)}  \label{lawRandomDrift}
\end{equation}%
which solves the fractional equation
\begin{equation*}
\left( \frac{\partial }{\partial t}+ a^{\alpha} \partial _{x}^{\alpha
}+\lambda (I-K)\right) u(x,t)=0,\quad x\geq 0,\;t>0
\end{equation*}%
subject to the initial condition $u_{0}=\delta $.
\end{coro}

\begin{proof}
From the fact that
\begin{align*}
Pr\{ N(t) + a \mathfrak{A}^\alpha_t < x \} =& \sum_{k=0}^\infty Pr\left\lbrace \mathfrak{A}^\alpha_t < \frac{x - k}{a} \right\rbrace p_k(t)\, \mathbf{1}_{(k < x)}
\end{align*}
we immediately get that
\begin{align*}
Pr\{ N(t) + a \mathfrak{A}^\alpha_t \in dx \}/dx = \frac{1}{a}\sum_{k=0}^\infty h_\alpha \left( \frac{x - k}{a}, t \right)p_k(t)\, \mathbf{1}_{(k < x)}
\end{align*}
which coincides with \eqref{lawRandomDrift} by the autosimilarity of the stable process $\mathfrak{A}^\alpha_t$, $t>0$.
\end{proof}


\begin{te}
\label{theo-H-bound-cond}
Let us consider $b\geq 0$, $\beta, \gamma \in (0,1]$. The density $h(x,t)=Pr\{\mathfrak{H}_{t}\in dx\}/dx$ of the
hitting time
\begin{equation}
\mathfrak{H}_{t}=\inf \left\{ s\geq 0\,:\,N(\mathfrak{A}^\gamma_s)+b\mathfrak{A}_{s}^{\beta }
\notin (0,t)\right\}, \quad t\geq 0 \label{hitting-mathfrak-h}
\end{equation}%
is the solution to the following equation
\begin{equation}
b^\beta \mathcal{D}^\beta_t \, u + \lambda^\gamma (I - K)^\gamma u = -
\frac{\partial u}{\partial x}, \quad x> 0, \; t> 0  \label{eq-galpha}
\end{equation}%
with initial and boundary conditions
\begin{equation}
\label{boundary-H}
u(x,0) = \delta(x), \qquad
u(0, t) = -\frac{\gamma \lambda^\gamma}{\Gamma(1-\gamma)}\sum_{k=0}^\infty \frac{\Gamma(k-\gamma)}{k!}  H(t-k)
\end{equation}
where $H(\cdot)$ is the Heaviside step function and
\begin{equation*}
K=\left\{
\begin{array}{ll}
e^{-\partial _{t}}, & \text{if }\;a>0, \\
B, & \text{if }\;a=0%
\end{array}%
\right.
\end{equation*}
is the shift operator.
\end{te}

\begin{os}
We observe that, for $\gamma=1$ in \eqref{boundary-H}, the boundary condition reduces to
\begin{equation*}
u(0, t) = \lambda \left[ H(t) - H(t-1) \right].
\end{equation*}
Indeed, for $\gamma \notin \mathbb{N}$, we can write
\begin{equation*}
u(0,t) = \lambda^\gamma \left[ H(t) - \gamma H(t-1) - \frac{\gamma (1-\gamma)}{2} H(t-2) - \frac{\gamma (1-\gamma)(2-\gamma)}{3!} H(t-3) - \ldots \right]
\end{equation*}
and, for $\gamma=1$, we get the claim.
\end{os}

\begin{proof}[Proof of Theorem \ref{theo-H-bound-cond}]
By definition, we can write%
\begin{align*}
\int_{0}^{+\infty }e^{-\xi t}\Pr \{\mathfrak{H}_{t}>x\}dt  = & \int_{0}^{+\infty }e^{-\xi t}Pr\{N(\mathfrak{A}^\gamma_x)+b\mathfrak{A}_{x}^{\beta }<t\}dt \\
=& \xi ^{-1}\int_{0}^{+\infty }e^{-\xi t}\frac{\partial }{\partial t} Pr\{N(\mathfrak{A}^\gamma_x)+b\mathfrak{A}_{x}^{\beta }<t\}dt \\
= & [\text{by \eqref{lap-gen-theorem}} ] \\
= & \xi ^{-1}e^{-b^{\beta }\xi ^{\beta }x-\lambda^\gamma x(1-e^{-\xi })^\gamma}.
\end{align*}%
Therefore we obtain
\begin{align}
\int_{0}^{+\infty }e^{-\xi t}h(x,t)dt = & - \int_0^{+\infty} e^{-\xi t} \frac{\partial}{\partial x}
Pr \{\mathfrak{H}_{t} >x\} dt \notag \\
= & \xi^{-1} \left[ b^{\beta }\xi^{\beta }+\lambda^\gamma (1-e^{-\xi })^\gamma \right] e^{-b^{\beta }\xi ^{\beta}x-\lambda^\gamma x(1-e^{-\xi })^\gamma} \notag \\
= & \widetilde{h}(x, \xi) \label{h-tilde}.
\end{align}%
We immediately get that
\begin{equation*}
\widetilde{\widetilde{h}}(\mu, \xi) = \int_0^\infty e^{-\mu x} \widetilde{h}(x, \xi)dx = \frac{\xi^{-1} \left[ b^{\beta}\xi^{\beta}+\lambda^\gamma (1-e^{-\xi })^\gamma \right]}{\mu + b^{\beta}\xi ^{\beta} +\lambda^\gamma (1-e^{-\xi})^\gamma}
\end{equation*}
Let us now focus on the equation \eqref{eq-galpha}. We have that
\begin{equation*}
\int_{0}^{+\infty }e^{-\xi t} \left[ b^{\beta} \mathcal{D}^\beta_t\, u(x,t)+\lambda^\gamma (I-K)^\gamma u(x,t) \right] dt = b^\beta \xi^{\beta} \widetilde{u}(x, \xi) - b^\beta \xi^{\beta -1} \delta(x) + \lambda^\gamma (1- e^{-\xi})^\gamma \widetilde{u}(x, \xi)
\end{equation*}
and therefore, equation \eqref{eq-galpha} takes the form
\begin{equation*}
b^\beta \xi^{\beta} \widetilde{u}(x, \xi) - b^\beta \xi^{\beta -1} \delta(x) + \lambda^\gamma (1- e^{-\xi})^\gamma \widetilde{u}(x, \xi) = - \frac{\partial \widetilde{u}}{\partial x}(x, \xi).
\end{equation*}
Furthermore,
\begin{equation*}
\int_0^\infty e^{-\mu x} \frac{\partial \widetilde{u}}{\partial x}(x, \xi) dx = \mu \widetilde{\widetilde{u}}(\mu, \xi) - \widetilde{u}(0, \xi)
\end{equation*}
where
\begin{equation*}
\widetilde{u}(0, \xi) = \int_0^\infty e^{-\xi t} u(0,t)dt = \xi^{-1} \lambda^\gamma \left(1 - e^{-\xi} \right)^\gamma.
\end{equation*}
Indeed, considering that
\begin{equation*}
-\frac{\gamma}{\Gamma(1-\gamma)}\sum_{k=0}^\infty \frac{\Gamma(k-\gamma)}{k!}  H(t-k) =  H(t) - \frac{\gamma}{\Gamma(1-\gamma)}\sum_{k=1}^\infty \frac{\Gamma(k-\gamma)}{k!}  H(t-k)
\end{equation*}
we can write
\begin{align*}
& \int_0^\infty e^{-\xi t} \left[ \lambda^\gamma H(t) -  \frac{\gamma \lambda^\gamma}{\Gamma(1-\gamma)} \sum_{k=1}^\infty  \frac{\Gamma(k-\gamma)}{k!}  H(t-k)  \right] dt\\
= & \frac{\lambda^\gamma}{\xi} - \frac{\gamma \lambda^\gamma}{\Gamma(1-\gamma)} \sum_{k=1}^\infty \frac{\Gamma(k-\gamma)}{k!}  \frac{e^{-\xi k}}{\xi}= \left[ \textrm{by \eqref{typical-bern-function}} \right]\\
= & \int_0^\infty \left( \frac{(1-e^{-s\lambda})}{\xi} \frac{1}{s^{\gamma+1}} - e^{-s\lambda} \sum_{k=1}^\infty \frac{s^{k-\gamma -1}}{k!} \frac{e^{-\xi k}}{\xi} \right) \frac{\gamma}{\Gamma(1-\gamma)} ds \\
= & \int_0^\infty \left( \frac{(1-e^{-s\lambda})}{\xi}  - e^{-s\lambda} \sum_{k=1}^\infty \frac{s^{k}}{k!} \frac{e^{-\xi k}}{\xi} \right) \frac{\gamma}{\Gamma(1-\gamma)}\frac{ds}{s^{\gamma+1}} \\
= & \int_0^\infty \left( \frac{1}{\xi}  - e^{-s\lambda} \sum_{k=0}^\infty \frac{s^{k}}{k!} \frac{e^{-\xi k}}{\xi} \right) \frac{\gamma}{\Gamma(1-\gamma)}\frac{ds}{s^{\gamma+1}} \\
= & \int_0^\infty \left( \frac{1 - e^{-s\lambda(1- e^{-\xi})} }{\xi} \right) \frac{\gamma}{\Gamma(1-\gamma)}\frac{ds}{s^{\gamma+1}}\\
= & \xi^{-1} \lambda^\gamma (1-e^{-\xi})^\gamma
\end{align*}
where we used once again formula \eqref{typical-bern-function}.

By collecting all pieces together, formula \eqref{eq-galpha} with initial and boundary conditions becomes
\begin{equation*}
b^\beta \xi^{\beta} \widetilde{\widetilde{u}}(\mu, \xi) - b^\beta \xi^{\beta -1} + \lambda^\gamma (1- e^{-\xi})^\gamma \widetilde{\widetilde{u}}(\mu, \xi) = - \mu \widetilde{\widetilde{u}}(\mu, \xi) + \xi^{-1} \lambda^\gamma \left(1 - e^{-\xi} \right)^\gamma.
\end{equation*}
Thus, we get that
\begin{equation}
\widetilde{\widetilde{u}}(\mu, \xi)  = \frac{b^\beta \xi^{\beta-1} + \xi^{-1}\lambda^\gamma (1-e^{-\xi})^\gamma }{\mu + b^\beta \xi^\beta + \lambda^\gamma (1- e^{-\xi})^\gamma}. \label{doub-lap-us}
\end{equation}
By observing that $\widetilde{\widetilde{u}} = \widetilde{\widetilde{h}}$, we get the claimed result.
\end{proof}


\section{L\'{e}vy processes with drifted Poisson time change}
\label{lastSect}

We consider now the L\'{e}vy process $X$ time-changed by an independent random time defined as in \eqref{N-time-changed-A}. This can be considered as a generalization of the result given in Lemma \ref{lemmaXN}.
\begin{te}
Let $X(t)$, $t\geq 0$, be the L\'{e}vy process previously introduced. Let $X_{j}$%
, $j=1,2,\ldots $ be i.i.d. random variables such that $X_{j}\sim X(1)$ for
all $j$. Then, for $\gamma ,\alpha \in (0,1]$, we have that
\begin{equation}
X(N(\mathfrak{A}_{t}^{\gamma })+a\mathfrak{A}_{t}^{\alpha })\overset{law}{=}%
\sum_{j=1}^{N(\mathfrak{A}_{t}^{\gamma })}X_{j}+X(a\mathfrak{A}_{t}^{\alpha
}),\quad t\geq 0,\;a\geq 0  \label{sub-X-and-pois}
\end{equation}%
where $\mathfrak{A}_{t}^{\gamma }$ and $\mathfrak{A}_{t}^{\alpha }$ are
independent stable subordinators. Furthermore, the infinitesimal generator
of \eqref{sub-X-and-pois} is written as
\begin{equation}
\mathcal{L}^{\alpha ,\gamma }f(x)=-(-a\mathcal{A})^{\alpha }f(x)-\lambda
^{\gamma }(I-K)^{\gamma }f(x)
\end{equation}%
where $K=e^{\mathcal{A}}$ is a shift operator and
\begin{equation}
-(-a\mathcal{A})^{\alpha }f(x)=\frac{\alpha a^{\alpha }}{\Gamma (1-\alpha )}%
\int_{0}^{\infty }\left( P_{s}f(x)-f(x)\right) \frac{ds}{s^{\alpha +1}}
\label{aA-symb}
\end{equation}%
with $P_{s}=e^{s\mathcal{A}}$, which is the semigroup of the L\'{e}vy
process $X(s)$, $s\geq 0$.
\end{te}

\begin{proof}
We first recall that $\mathfrak{A}^1_t = t$ is the elementary subordinator. Thus, for $\gamma=\alpha=1$, the characteristic function of the right-hand side of \eqref{sub-X-and-pois} is given by
\begin{align*}
\mathbb{E} \exp \left(i \xi \sum_{j=1}^{N(t)} X_j + i \xi  X(at)  \right) = & e^{-at \Psi(\xi)} \mathbb{E}\exp \left( i\xi X(1)N(t) \right)\\
= & e^{-at \Psi(\xi)} \mathbb{E}\left( \mathbb{E} e^{i\xi X(1)} \right)^{N(t)} =  e^{-at \Psi(\xi)} \mathbb{E}\left( e^{-\Psi(\xi)} \right)^{N(t)}\\
= & e^{-at \Psi(\xi)} \exp \left(- \lambda t \left(1- e^{-\Psi(\xi)}\right) \right) =  \exp\left( -t \Phi(\xi) \right).
\end{align*}
which coincides with \eqref{Phi-symbol}. From the uniqueness of the Laplace transform
\begin{equation*}
\mathbb{E} \exp\left( - g(\Phi(\xi)) \mathfrak{A} \right)
\end{equation*}
(for some well-behaved $g$) we obtain the equality in distribution \eqref{sub-X-and-pois}. Indeed, 
\begin{align*}
- \partial_t \, \mathbb{E} \exp \left(-a \Psi(\xi) \mathfrak{A}^\alpha_t - \lambda \left(1- e^{-\Psi(\xi)}\right) \mathfrak{A}^\gamma_t \right) \Big|_{t=0} = \left( a \Psi(\xi) \right)^\alpha + \lambda^\gamma \left( 1 - e^{-\Psi(\xi)}\right)^\gamma
\end{align*}
is the Fourier symbol of the process which appears in the right-hand side of \eqref{sub-X-and-pois}. Let us write the Fourier symbol as
\begin{equation}
\label{symb-gen-space-diff}
g_{\alpha, \gamma}(\xi) = \left( a \Psi(\xi) \right)^\alpha + \lambda^\gamma \left( 1 - e^{-\Psi(\xi)}\right)^\gamma.
\end{equation}
We now show that $-g_{\alpha, \gamma}(\xi)$ is  the Fourier multiplier of the infinitesimal generator of the left-hand side of \eqref{sub-X-and-pois}. The Fourier transform of \eqref{aA-symb} is given by
\begin{align*}
 \frac{ \alpha a^\alpha }{\Gamma(1-\alpha)} \int_{0}^\infty \left( e^{-s\Psi(\xi)} - 1 \right) \frac{ds}{s^{\alpha +1}} \widehat{f}(\xi) = \left( a \Psi(\xi) \right)^\alpha \, \widehat{f}(\xi),
\end{align*}
where we recall that $e^{-s\Psi(\xi)}$ is the symbol of the semigroup $P_s$ associated to the infinitesimal generator $\mathcal{A}$ and
\begin{equation}
M(ds)= \frac{\alpha}{\Gamma(1-\alpha)}\frac{ds}{s^{\alpha+1}}
\end{equation}
is the L\'{e}vy measure of a stable subordinator of order $\alpha \in (0,1)$. As we have shown before, we also have that
\begin{equation*}
\int_\mathbb{R} e^{i\xi x} \lambda^\gamma (I-K)^\gamma f(x) dx = \lambda^\gamma (1 - e^{-\xi})^\gamma \widehat{f}(\xi)
\end{equation*}
iff $K=e^{-\partial_x}$ is the translation operator.

Now we show that
\begin{equation}
\widehat{(I-e^\mathcal{A})^\gamma f} = (1-e^{-\Psi(\xi)})^\gamma \widehat{f}.\label{symb-gamma-pow}
\end{equation}
By \eqref{I-K-power}, we get the Fourier transform
\begin{equation*}
\widehat{(I-K)^\gamma f}= \sum_{j=0}^\infty (-1)^j \binom{\gamma}{j} \widehat{K^j\, f}.
\end{equation*}
For $K=e^\mathcal{A}=P_1$, where $P_t$ is the semigroup with symbol $e^{-t \Psi}$, we obtain  \eqref{symb-gamma-pow}.
\end{proof}

\begin{os}
We observe that $K=e^\mathcal{A}$ is a translation operator. Furthermore, it
represents a Frobenious-Perron operator associated with the transformation $%
 x\mapsto x-f(x)$ where $f(x)$ is a random jump with generator $\mathcal{A}
$. If $\mathcal{A}= -\partial_x$, then the jump equals $f=1$.
\end{os}

Moreover we notice that, for $\alpha=\gamma=1$ and $a\geq 0$, we have that
\begin{equation*}
X(N(t) + at) \overset{law}{=} \sum_{j=1}^{N(t)} X_j + X(at)
\end{equation*}
where $X_j \sim X(1)$ are independent for all $j$ and the process $X(t)$, $%
t\geq 0$ is governed by the equation
\begin{equation*}
\frac{\partial u}{\partial t} = \mathcal{A}u.
\end{equation*}

As a different time argument we consider now the hitting time defined in \eqref{hitting-mathfrak-h} and thus we apply the result of Theorem \ref{theo-H-bound-cond} in order to define a new time-changed L\'{e}vy process.

\begin{te}
\label{theo-quasi-last}
Let $X(t)$, $t>0$ be a L\'{e}vy process with symbol \eqref{Levy-symb-X} independent from $\mathfrak{H}_t$, $t>0$. The governing equation of
\begin{equation}
X(\mathfrak{H}_t), \quad t>0
\end{equation}
is given by
\begin{equation}
\label{gov-eq-last1}
b^\beta \mathcal{D}^\beta_t\, u(x,t) + \lambda^{\gamma} (I-K_1)^{\gamma} u(x,t) = \mathcal{A} u(x,t), \quad x \in \mathbb{R}, \; t>0
\end{equation}
subject to the initial and boundary conditions \eqref{boundary-H}.
\end{te}
\begin{proof}
Let us consider the double Laplace transform \eqref{doub-lap-us}. Since $-\Psi$ is the Fourier multiplier of $\mathcal{A}$, the Fourier transform of \eqref{gov-eq-last1} is written as
\begin{equation}
\label{gov-eq-last1-fourier}
b^\beta \mathcal{D}^\beta_t\, \widehat{u}(\mu, t) + \lambda^{\gamma} (I-K_1)^{\gamma} \widehat{u}(\mu, t) = - \Psi(\mu) \widehat{u}(\mu, t).
\end{equation}
By passing to the Laplace transform of \eqref{gov-eq-last1-fourier} and following the same arguments as in the proof of Theorem \ref{theo-H-bound-cond}, we get that formula \eqref{doub-lap-us} leads to
\begin{equation}
\widetilde{\widehat{u}}(\mu, \xi)  = \frac{b^\beta \xi^{\beta-1} + \xi^{-1}\lambda^\gamma (1-e^{-\xi})^\gamma }{\Psi(\mu) + b^\beta \xi^\beta + \lambda^\gamma (1- e^{-\xi})^\gamma} \label{doub-lap-four-us}
\end{equation}
which is the Laplace-Fourier transform of the solution to \eqref{gov-eq-last1} subject to the conditions \eqref{boundary-H}. With the Laplace transform \eqref{h-tilde} at hand, we also get that
\begin{align}
\label{lap-transf-H-last}
\int_0^\infty e^{-\xi t} \mathbb{E}e^{i\mu X(\mathfrak{H}_t)} dt = & \int_0^\infty e^{-\xi t} \mathbb{E}e^{- \mathfrak{H}_t\, \Psi(\mu) } dt =  \int_0^\infty e^{- x \Psi(\mu)} \widetilde{h}(x, \xi) dx
\end{align}
equals \eqref{doub-lap-four-us} and therefore we obtain the claimed result.
\end{proof}

Finally, as a further generalization, we write
\begin{equation}
\label{last-proc1}
\mathfrak{E}^{(j)}_t = N_j(\mathfrak{A}^{\gamma_j}_t) + b_j\, \mathfrak{A}^{\theta_j}_t, \quad t>0, \qquad j=1,2
\end{equation}
with $\gamma_j, \theta_j \in (0,1]$ for all $j$ and
\begin{equation}
\label{last-proc2}
\mathfrak{H}^{(j)}_t = \inf\{s\geq 0\, :\, \mathfrak{E}^{(j)}_s \notin (0,t))\}, \quad t>0, \qquad j=1,2
\end{equation}
where $b_j\geq 0$, $N_j(t)$, $t>0$ is a Poisson process with rate $\lambda_j>0$, $j=1,2$, whereas, we still denote by $\mathfrak{H}_t$ the hitting time \eqref{hitting-mathfrak-h}. All the processes are independent from each other. Moreover,  we consider here the processes \eqref{last-proc1} and \eqref{last-proc2} with $\theta_1=\beta$ and $\theta_2=\alpha$ to streamline the notation.

\begin{te}
Let $X(t)$, $t>0$ be a L\'{e}vy process with symbol \eqref{Levy-symb-X} independent from $\mathfrak{E}^{(1)}_t$, $t>0$ and $\mathfrak{H}^{(2)}_t$, $t>0$. The governing equation of the process
\begin{equation}
X(\mathfrak{E}^{(2)}_{\mathfrak{H}^{(1)}_t}), \quad t>0
\end{equation}
is written as
\begin{equation}
\label{eq-gov-last-comp}
\left( b_1^\beta \mathcal{D}^\beta_t\,  + \lambda_1^{\gamma_1} (I-K_1)^{\gamma_1}  + (-b_2 \mathcal{A})^\alpha + \lambda_2^{\gamma_2} (I-K_2)^{\gamma_2}\right) u(x,t) = 0, \quad x\in \mathbb{R}, \; t>0
\end{equation}
where $K_1=e^{-\partial_t}$ and $K_2=e^{\mathcal{A}}$, subject to the initial and boundary conditions \eqref{boundary-H}.
\end{te}
\begin{proof}
We start once again from \eqref{doub-lap-us}. By considering the Fourier transform of \eqref{eq-gov-last-comp}, from the previous results and by formula \eqref{symb-gen-space-diff}, we get that
\begin{align*}
g_{\alpha, \gamma_2}(\mu) = & \int_\mathbb{R} e^{i\mu x} \bigg[-(-b_2 \mathcal{A})^\alpha u(x,t) - \lambda_2^{\gamma_2} (I-K_2)^{\gamma_2} u(x,t) \bigg]dx\\
= & b_2^\alpha \left( \Psi(\mu)\right)^\alpha + \lambda_2^{\gamma_2} \left( 1- e^{-\Psi(\mu)} \right)^{\gamma_2}
\end{align*}
and (see the proof of the previous theorem)
\begin{equation*}
\widetilde{\widehat{u}}(\mu, \xi)  = \frac{b_1^\beta \xi^{\beta-1} + \xi^{-1}\lambda_1^{\gamma_1} (1-e^{-\xi})^{\gamma_1} }{g_{\alpha, \gamma_2}(\mu) + b_1^\beta \xi^\beta + \lambda_1^{\gamma_1} (1- e^{-\xi})^{\gamma_1}}
\end{equation*}
which is the Laplace-Fourier transform of the solution to \eqref{eq-gov-last-comp} subject to the conditions \eqref{boundary-H}. Now, it remains to see that
\begin{align*}
\mathbb{E} \exp\left( i \mu X(\mathfrak{E}^{(2)}_{\mathfrak{H}^{(1)}_t}) \right) = & \mathbb{E} \exp\left( - (\mathfrak{E}^{(2)}_{\mathfrak{H}^{(1)}_t})\, \Psi(\mu) \right) \\
= & [\textrm{see formula \eqref{lap-gen-theorem} with } a=b_2, \, \lambda=\lambda_2,\, \gamma=\gamma_2,\, \beta=1] \\
= & \mathbb{E} \exp\left( - (\mathfrak{H}^{(1)}_t) \left( b_2^\alpha (\Psi(\mu))^\alpha + \lambda_2 (1- e^{-\Psi(\mu)})^{\gamma_2} \right)  \right)\\
= & \int_0^\infty \exp \left(- x g_{\alpha, \gamma_2}(\mu) \right)\, Pr\{ \mathfrak{H}^{(1)}_t \in dx \} \\
= & \int_0^\infty \exp \left(- x g_{\alpha, \gamma_2}(\mu) \right)\, h(x,t)dx.
\end{align*}
By considering the Laplace transform $\widetilde{h}(x, \xi) = \int_0^\infty e^{-\xi t} h(x,t)dt$ we get that
\begin{equation*}
\int_0^\infty e^{-\xi t} \mathbb{E} \exp\left( i \mu X(\mathfrak{E}^{(2)}_{\mathfrak{H}^{(1)}_t}) \right) \, dt = \int_0^\infty \exp \left(- x g_{\alpha, \gamma_2}(\mu) \right) \, \widetilde{h}(x, \xi) dx
\end{equation*}
as in the Laplace transform \eqref{lap-transf-H-last}. Therefore, by the same arguments as in the proof of Theorem \ref{theo-quasi-last} we conclude the proof.
\end{proof}

\end{document}